\pdfoutput=1
\documentclass{article}
\usepackage{amssymb}
\usepackage{amsmath}
\usepackage[table]{xcolor}
\usepackage{float}
\usepackage{diagbox}
\usepackage{graphicx}
\usepackage{amsthm}
\usepackage{enumerate}
\usepackage{pxfonts}
\usepackage{xassoccnt}
\usepackage{nameref}
\usepackage[english]{babel}
\usepackage[numbers]{natbib}
\usepackage{hyperref}
\usepackage{cleveref}
\usepackage{mathrsfs}
\usepackage{thmtools}
\usepackage{thm-restate}
\usepackage{tikz}
\usetikzlibrary{decorations.pathreplacing,calligraphy}
\usepackage{tikz-cd}
\usetikzlibrary{automata, arrows.meta, positioning}
\usepackage{makecell}
\usepackage{lineno}
\usetikzlibrary{calc}
\usepackage{relsize}
\tikzset{fontscale/.style = {font=\relsize{#1}}}

\newtheorem{theorem}{Theorem}[section]
\crefname{theorem}{theorem}{theorems}
\newtheorem{proposition}{Proposition}[section]
\crefname{proposition}{proposition}{propositions}
\newtheorem{lemma}{Lemma}[section]
\crefname{lemma}{lemma}{lemmas}
\newtheorem{corollary}{Corollary}[section]
\crefname{corollary}{Corollary}{corollaries}
\newtheorem{remark}{Remark}[section]
\crefname{remark}{remark}{remarks}
\newtheorem{example}{Example}[section]
\crefname{example}{example}{examples}
\newtheorem{definition}{Definition}[section]
\crefname{definition}{definition}{definitions}

\DeclareCoupledCountersGroup{theorems}
\DeclareCoupledCounters[name=theorems]{theorem,proposition,lemma,remark,example,corollary,definition}

\begin{document}

\title{On the minimal components of substitution subshifts}
\author{Raphaël Henry\footnote{Aix Marseille Univ, CNRS, I2M, 3 place Victor Hugo, Case 19, 13331 Marseille Cedex 3, France}}
\date{Published in Theoretical Computer Science \url{https://doi.org/10.1016/j.tcs.2025.115517}}

\maketitle

\pagenumbering{arabic}
\setcounter{page}{1}

\begin{abstract}
	In this paper we study substitutions on $A^\mathbb{Z}$ where $A$ is a finite alphabet. We precisely characterize the minimal components of substitution subshifts, give an optimal bound for their number and describe their dynamics. The explicitness of these results provides a method to algorithmically compute and count the minimal components of a given substitution subshift.
\end{abstract}

\section*{Introduction}

In symbolic dynamics, substitutions on a finite alphabet $A$ are a rich way to generate infinite words, for a general background see \cite{AS} or the upcoming book \cite{BDP}. In particular, \textit{substitution subshifts} on $A^\mathbb{Z}$ have been extensively studied, for a general overview see \cite{Queffelec} and for various computability results see \cite{BPR2}.

A very natural property of subshifts is \textit{minimality}, meaning that the orbit of every point is dense. Many results were proven in the framework of minimal subshifts, for instance, Damanik and Lenz \cite{DL} proved that a substitution subshift is minimal if and only if it is linearly repetitive. As a simple way to construct minimal subshifts, it is well-known that \textit{primitive} substitutions generate minimal subshifts (see for example \cite{Queffelec}), like the Fibonacci substitution $\sigma_F : 0 \mapsto 01, 1 \mapsto 0$. Some non-primitive substitutions also generate minimal subshifts, like the Chacon substitution $\sigma_C : 0 \mapsto 0010, 1 \mapsto 1$, which is even \textit{non-growing} because the letter $1$ is \textit{bounded}. Acknowledging that it is crucial to distinguish growing and bounded letters, Shimomura \cite{Shimomura} introduced a weaker version of primitivity, called \textit{l-primitivity}, that means being primitive when removing the bounded letters. Moreover, in a recent interest towards non-minimal substitution subshifts, Maloney and Rust \cite{MR} identified that \textit{tameness}, a combinatorial property of the substitution, is essential for minimality (this \textit{tame} vocabulary must not be confused with a property of general dynamical systems related to semigroups, which is not related to our topic). The idea is that a \textit{wild} (i.e., not tame) substitution will generate an infinite periodic word, which might prevent minimality. In his same article, Shimomura showed that these two properties characterize the minimal substitution subshifts: a substitution subshift $X_\sigma$ is minimal if and only if there exists a tame and l-primitive substitution $\sigma'$ such that $X_\sigma = X_{\sigma'}$.

Now that minimal substitution subshifts are well-understood, the next step is to describe the structure of non-minimal substitution subshifts. In this direction, Béal, Perrin and Restivo \cite{BPR2} showed that any substitution subshift is \textit{quasi-minimal}, meaning that it has a finite number of subshifts; and Bezuglyi, Kwiatowski and Medynets \cite{BKM} showed that any growing substitution subshift on $d$ letters has at most $d$ minimal components. The number of subshifts - or the number of minimal components - can then be seen as a measure of the complexity of the subshift, and this paper is dedicated to the study of the minimal components of substitution subshifts.

For growing substitutions, we have encouraging results. The upper bound of their number in \cite{BKM} is already important, but it does not describe the minimal components themselves. On that note, one known family of minimal components is given by what Durand called \textit{main sub-substitutions}: they are primitive substitutions obtained by restricting a suitable power of the original substitution to a suitable subalphabet. Each of them generates a minimal component, and they are thought to generate all the minimal components. To support this idea, the analog of main sub-substitutions was introduced by Cortez and Solomyak \cite{CS} in the framework of tiling substitutions on $\mathbb{R}^d$, and they show that the minimal components are exactly the subshifts generated by these sub-substitutions. In particular, it is remarkable that every minimal component is itself a substitution system.

For non-growing substitutions, another family of minimal components arises from wildness. The idea, mentionned earlier, originates from the closely related framework of D0L-systems, where Ehrenfeucht and Rozenberg \cite{ER} begin the investigation of infinite repetitions in their language and Klouda and Starosta \cite{KS} provide effective results. More precisely, if a substitution is wild, then there exists a growing letter that generates a periodic word over bounded letters to its left or to its right. Notably, this was already used by Pansiot \cite{Pansiot} to characterize purely morphic words with factor complexity $\Theta(n^2)$. With the work of Maloney and Rust \cite{MR} and Shimomura \cite{Shimomura}, this transposes to the framework of substitution subshifts  where, in particular, each such periodic word provides a minimal component.

The main contribution of this paper is the characterization of the minimal components of any substitution subshift. In a first time, we generalize main sub-substitutions to be tame and l-primitive substitutions obtained by restricting a suitable power of the original substitution to a subalphabet. Our definition, however, is more effective than the one of Durand as it relies on finding the suitable subalphabets by constructing a specific oriented graph. Thanks to Shimomura's result, the main sub-substitutions provide minimal components containing growing letters. We then use the fact that the substitution induces a permutation on the subshifts to show that these components are exactly all the minimal components containing growing letters, which we call \textit{tame minimal components}. In a second time, we build on results of \cite{KS} and \cite{BPR2} to define precisely and effectively the periodic words over bounded letters produced by a wild substitution. In parallel, we adapt the notion of 1-blocks introduced in \cite{Devyatov} to obtain a decomposition of the words over bounded letters. This shows that these single periodic orbits are exactly all the minimal components over bounded letters, which we call \textit{wild minimal components}. To complete the structure of the minimal components, we also describe how the substitution acts on them: we show that the substitution induces a permutation on the minimal components that can be read from the directed graphs introduced in our proof.

In our effort to make the characterization as effective as possible, we obtain that the number of minimal components of a given substitution subshift is computable. This is paired with a Python implementation of the computation of minimal components at \url{https://codeberg.org/RaphaelHENRY/MinimalComponents.git}. We are also able to bound the number of minimal components as a function of the alphabet size, generalizing the previously known bound to the non-growing case. Finally, as a way to apply our results to various examples, we compute the minimal components for all substitutions on two letters.

We conclude this paper by opening the study to what we call \textit{irreducible components}
and briefly discussing potential generalizations of our results to other types of subshifts.

\section{Preliminaries}\label{prelim}

\subsection{Notations}

\subsubsection*{Words}

Given a finite alphabet $A$, $A^*$ denotes the set of all words on $A$, in particular $\varepsilon$ is the empty word and we define $A^+:=A^*\backslash\{\varepsilon\}$. For $u=a_0 a_1...a_{n-1} \in A^*$, we write $\lvert u \rvert := n$ the \textit{length} of $u$. To simplify some expressions, we will use the following notation: if $(u_i)_{0 \leq i \leq n-1} \in (A^*)^n$, we write the concatenation from left to right as
\begin{align}\label{concat}
	\displaystyle \bigsqcup_{i=0}^{n-1} u_i := u_0 u_1...u_{n-2}u_{n-1} ~~~~~~~~& \textrm{with the convention that } \displaystyle \bigsqcup_{i=0}^{-1} u_i := \varepsilon.
\end{align}

If words $u,u',v,w \in A^*$ are such that $w=uvu'$, we say that $v$ is a \textit{factor} of $w$ and we write $v \sqsubset w$. If $u=\varepsilon$, $v$ is a \textit{prefix} of $w$, and if $u'=\varepsilon$, $v$ is a \textit{suffix} of $w$. If $u,v \in A^*$ and there exists $w,w' \in A^*$ such that $u=ww'$ and $v=w'w$, we say that $u$ is a \textit{cyclic shift} of $v$. A word $u \in A^+$ is \textit{primitive} if $u=v^k$ for some $v \in A^+$ and $k \geq 1$ implies that $k=1$. The shortest word $v$ such that $u=v^k$ for some $k \geq 1$ is a primitive word called the \textit{primitive root} of $u$.

Let $A^{\mathbb{Z}}$ be the set of all two-sided sequences of letters of $A$. For $x \in A^{\mathbb{Z}}$ and $i-1 \leq j \in \mathbb{Z}$, the word $x_i...x_j$ is a \textit{factor} of $x$ and we write $x_i...x_j \sqsubset x$, with the convention that $x_i...x_{i-1} = \varepsilon$. If $u \in A^+$, we write the infinite word to the right $u^\omega := uuu...$ and the infinite word to the left ${}^\omega u := ...uuu$. By writing a dot right to the left of the letter of index 0, we define the bi-infinite words ${}^\omega u^\omega := {}^\omega u.u^\omega$ and ${}^\omega u v w^\omega := {}^\omega u.vw^\omega$.

\subsubsection*{Substitutions}

A map $\varphi : A^* \rightarrow A^*$ is a \textit{morphism} if for all $u,v \in A^*$, $\varphi(uv) = \varphi(u)\varphi(v)$. In this case, we write $\varphi : A \rightarrow A^*$ and the language of $\varphi$ is the set
\begin{equation*}
	\mathscr L(\varphi) := \left\{u \in A^* ~\middle|~ \exists a \in A, \exists n \geq 0, u \sqsubset \varphi^n(a)\right\}.
\end{equation*}
If a morphism $\sigma$ is \textit{non-erasing}, i.e., for all $a \in A$, $\sigma(a) \neq \varepsilon$, then we say it is a \textit{substitution} instead and we write $\sigma : A \rightarrow A^+$.

We say that a substitution $\sigma : A \rightarrow A^+$ is \textit{primitive} if there exists $m \geq 1$ such that, for every $a,b \in A$, $a \sqsubset \sigma^m(b)$. An equivalent definition is asking the incidence matrix of the substitution (that is, the matrix with coefficients $m_{a,b} = \left\lvert\sigma(b)\right\rvert_a$ for $a,b \in A$) to be primitive, but in this paper we will mostly use the combinatorial definition.

If $D \subset A$ is a subalphabet such that $\sigma(D) \subset D^+$, we have the restriction $\sigma|_D : D \rightarrow D^+$.

\subsubsection*{Bounded and growing letters}

Let $\sigma : A \rightarrow A^+$ is a substitution. If $a \in A$ is such that the sequence $\left( \left\lvert \sigma^n(a) \right\vert \right)_{n \geq 0}$ is bounded, we say that $a$ is \textit{bounded}, otherwise we say that $a$ is \textit{growing}. Once the substitution is fixed, $B$ always denotes the set of bounded letters, and $C$ always denotes the set of growing letters. In particular, if $a \in B^+$, then $\sigma(a) \in B^+$, and if $a \in C$, then $\sigma(a)$ contains a growing letter.

If $B=\emptyset$, we say that $\sigma$ is \textit{growing}. If $u \in A^*$, we define $alph_C(u) := \left\{c \in C ~\middle|~ c \sqsubset u\right\}$. If $X \subset A^\mathbb{Z}$, we define the alphabet of growing letters in $X$
\begin{equation*}
	alph_C(X) := \left\{c \in C ~\middle|~ \exists x \in X, c \sqsubset x\right\}.
\end{equation*}

\subsubsection*{Subshifts}

We define the left shift $T : A^{\mathbb{Z}} \rightarrow A^{\mathbb{Z}}$ where $T(x)_i = x_{i+1}$. A non-empty set $X \subset A^{\mathbb{Z}}$ is a \textit{subshift} if it is closed for the prodiscrete topology on $A^\mathbb{Z}$ and satisfies $T(X)=X$. If $x \in A^\mathbb{Z}$, the smallest subshift containing $x$ is the set
\begin{equation*}
	X(x) := \overline{\left\{T^k(x) \mid k \in \mathbb{Z}\right\}}.
\end{equation*}
In particular, if $u \in A^+$, $X({}^\omega u^\omega)$ is a single periodic orbit. The language of the subshift $X$ is the set
\begin{equation*}
	\mathscr L(X) := \left\{u \in A^* ~\middle|~ \exists x \in X, u \sqsubset x\right\}.
\end{equation*}

A word $u$ is said to occur with \textit{bounded gaps} in $X$ if there exists $L \geq 1$ such that, for all $v \in \mathscr L(X)$, if $\lvert v \rvert \geq L$ then $u \sqsubset v$.

Given a substitution $\sigma : A \rightarrow A^+$ where $C \neq \emptyset$, the associated \textit{substitution subshift} is the set
\begin{equation*}
	X_{\sigma} := \left\{x \in A^\mathbb{Z} ~\middle|~ \forall u \sqsubset x, u \in \mathscr L(\sigma)\right\}.
\end{equation*}
In particular, $\mathscr L(X_\sigma) \subset \mathscr L(\sigma)$. If $A \subset \mathscr L(X_\sigma)$, $\sigma$ is \textit{admissible}.

\begin{example}\label{ex1}
	The substitution $\sigma : 0 \mapsto 12, 1 \mapsto 22, 2 \mapsto 11$ is not admissible because $X_\sigma = X({}^\omega 12^\omega) \cup X({}^\omega 21^\omega)$ and $0 \notin \mathscr L(X_\sigma)$.
\end{example}

\subsubsection*{Minimality}

A subshift $X$ is \textit{minimal} if it does not contain any subshift other than itself.

\begin{remark}\label{minimal}
	If $X$ is a subshift, then the three following properties are equivalent:
	
	(i) $X$ is minimal.
	
	(ii) For every $x \in X$, $X(x) = X$.
	
	(iii) Every word of $\mathscr L(X)$ occurs with bounded gaps in $X$.
	
	A proof can be found in \cite{Pytheas} for example.
\end{remark}

If $X$ is a subshift, a minimal subshift $Y \subset X$ is called a \textit{minimal component} of $X$. If $X$ has a unique minimal component, we say it is \textit{essentially minimal}, which is not equivalent to being minimal as illustrated with \Cref{ex2}. More generally, if $X$ contains a finite number of subshifts, we say it is \textit{quasi-minimal}. In particular, every substitution subshift is quasi-minimal.

\begin{proposition}[{\cite[Proposition 10.8]{BPR2}}]\label{quasiminimal}
	Let $\varphi : A \rightarrow A^*$ be a morphism. Then $X_\varphi$ is quasi-minimal.
\end{proposition}

Note that the same result had already been shown with various assumptions on the substitution, see \cite[Lemma 5.13]{MR} or \cite[Proposition 5.14]{BSTY}.

\subsubsection*{D0L-systems}

$L$-systems are a large class of objects that define formal languages, they were initially introduced by Lindenmayer to study mathematically the development of simple filamentous organisms. For a general overview one can refer to \cite{RS}.

In particular, a D0L-system is a triplet $G = (A, \varphi, w)$ where $A$ is a finite alphabet, $\varphi : A \rightarrow A^*$ is a morphism and $w$ is a word in $A^+$; and the associated language is the set
\begin{equation*}
	\mathscr L(G) := \left\{u \in A^* ~\middle|~ \exists n \geq 0, u \sqsubset \varphi^n(w) \right\}.
\end{equation*}

\subsection{Minimality and beyond}

\subsubsection{Characterization of minimality}

The minimality of substitution subshifts relies on two properties of the substitution, the first one being a weaker version of primitivity.

\begin{definition}
	We say that a substitution $\sigma$ is \textit{l-primitive} if there exists $n \geq 1$ such that, for all $a,b \in C$, $a \sqsubset \sigma^n(b)$.
\end{definition}

Note that for growing substitutions, l-primitivity is the same as primitivity.

\begin{remark}\label{enough}
	The l-primitivity of a substitution depends only on its incidence matrix, but it is not sufficient to ensure minimality, because an l-primitive substitution that generates a minimal subshift can have the same incidence matrix as an l-primitive substitution that does not generate a minimal subshift. This is illustrated by the following example.
\end{remark}

\begin{example}\label{exWild}
	Consider the Chacon substitution $\sigma : 0 \mapsto 0010, 1 \mapsto 1$ and the substitution $\sigma' : 0 \mapsto 0001, 1 \mapsto 1$. For both we have $C = \{0\}$ and $B = \{1\}$, they have the same incidence matrix and they are both l-primitive. However, $X_\sigma$ is known to be minimal but $\{{}^\omega 1^\omega\} \subsetneq X_{\sigma'}$ so $X_{\sigma'}$ is not minimal.
\end{example}

In order to avoid the situation described in \Cref{enough}, a second property is needed. First, if $\sigma : A \rightarrow A^+$ is a substitution, we say that a letter $c \in C$ is \textit{left-isolated} (resp. \textit{right-isolated}) if there exist $n \geq 1$, $u \in B^+$ and $v \in A^*$ such that $\sigma^n(c) = ucv$ (resp. $\sigma^n(c) = vcu$). We define $C_{liso}$ (resp. $C_{riso}$) the set of left-isolated (resp. right-isolated) letters, and $C_{niso}$ the set of growing letters that are neither left- nor right-isolated.

\begin{definition}
	We say that a substitution $\sigma$ is \textit{tame} if $C_{liso} = C_{riso} = \emptyset$, otherwise we say that $\sigma$ is \textit{wild}.
\end{definition}

Note that growing substitutions are always tame.

\begin{example}
	Following \Cref{exWild}, we have that $\sigma$ is tame and $\sigma'$ is wild.
\end{example}

Now we have all the tools to characterize the minimal substitution subshifts.

\begin{theorem}[{\cite[Theorem A]{Shimomura}}]\label{Shimomura}
	Let $\sigma : A \rightarrow A^+$ be a substitution.
	
	(i) Suppose that $\sigma$ is tame and l-primitive. Then, $X_{\sigma}$ is minimal.
	
	(ii) Suppose that $X_{\sigma}$ is minimal and is not a single periodic orbit. Then, $\sigma$ is tame, and there exists a unique subalphabet $D \subset A$ and a restriction $\sigma|_D : D \rightarrow D^+$ such that $\sigma|_D$ is tame, l-primitive, and $X_\sigma=X_{\sigma|_D}$.
\end{theorem}

As single periodic orbits are minimal and can be expressed as a substitution subshift, this result gives an equivalent class for minimal substitution subshifts.

\begin{corollary}[{\cite[Corollary B]{Shimomura}}]
	Let $\mathscr M$ be the class of all minimal substitution subshifts. Let $\mathscr M'$ be the class of all $X_\sigma$ such that $\sigma$ is a tame and l-primitive substitution. Then, it follows that $\mathscr M = \mathscr M'$.
\end{corollary}

\subsubsection{Minimal components for growing substitutions}

When $\sigma : A \rightarrow A^+$ is a growing substitution, Durand \cite{Durand} defines \textit{main sub-substitutions}. As one goal of this paper is to generalize them, we will describe them only briefly.

By raising the incidence matrix of $\sigma$ to a suitable power $p$, Durand introduces \textit{principal primitive components} of $A$, that are the disjoint alphabets $A_i \subset A$ such that $\sigma^p(A_i) \subset A_i^+$ and the restriction $\sigma^p|_{A_i}$ is primitive. The $\sigma^p|_{A_i}$ are called \textit{main sub-substitutions} and they generate the minimal components $X_{\sigma^p|_{A_i}}$.

\begin{example}
	Consider the growing substitution $\sigma : 0 \mapsto 11, 1 \mapsto 00$. Looking at $\sigma^2$, its main sub-substitutions are $\sigma^2|_{\{0\}}$ and $\sigma^2|_{\{1\}}$, so $X_{\sigma^2|_{\{0\}}} = \{{}^\omega 0^\omega\}$ and $X_{\sigma^2|_{\{1\}}} = \{{}^\omega 1^\omega\}$ are minimal components of $X_\sigma$. Note that, in this example, these are the only minimal components.
\end{example}

More generally, we have an upper bound of the number of minimal components that does not require to describe them.

\begin{proposition}[{\cite[Remark 5.7]{BKM}}]\label{growingBound}
	Let $\sigma : A \rightarrow A^+$ be a growing substitution. Then $X_\sigma$ has at most $\lvert A \rvert$ minimal components.
\end{proposition}

Let us draw a parallel with a result from Cortez and Solomyak \cite{CS} in the framework of admissible substitution tiling spaces on $\mathbb{R}^d$.

\begin{lemma}[{\cite[Lemma 2.9]{CS}}]\label{Xphik}
	Let $\omega$ be an admissible tile substitution. For all $k \geq 2$, $X_{\omega^k} = X_\omega$.
\end{lemma}

This allows to replace $\omega$ by a suitable power $\omega^p$ such that its incidence matrix provides alphabets $A_i$ and primitive substitutions $\omega^p|_{A_i}$ like Durand. What is remarkable in this framework is that this describes all the minimal components.

\begin{proposition}[{\cite[Lemma 2.10 (i)]{CS}}]\label{Cortez}
	Let $\omega$ be an admissible tile substitution that provides primitive substitutions $\omega|_{A_i}$. Then the minimal components of $X_\omega$ are the $X_{\omega|_{A_i}}$.
\end{proposition}

As the $A_i$ are pairwise disjoint, we get that any admissible tiling substitution system on a set of tiles $A$ has at most $\lvert A \rvert$ minimal components, as in \Cref{growingBound}.

\subsubsection{Minimal components for non-growing substitutions}

When a substitution is not growing, the production of bounded letters is closely related to tameness.

\begin{proposition}[{\cite[Theorem 2.9]{MR}}, {\cite[Proposition 3.17]{Shimomura}}]\label{Lbounded}
	Let $\sigma : A \rightarrow A^+$ be a substitution. Then the words in $\mathscr L(X_\sigma) \cap B^*$ have bounded length if and only $\sigma$ is tame.
\end{proposition}

\begin{remark}
	A similar result was proven earlier for the language of D0L-systems in \cite[proof of Lemma 2.1]{ER} with a different vocabulary: tameness is called the \textit{edge condition}, and the fact that $\mathscr L(G) \cap B^*$ is infinite is called \textit{pushyness}.
\end{remark}

\Cref{Lbounded} relies in part on the following lemma.

\begin{lemma}[{\cite[Lemma 2.8]{MR}},{\cite[Lemma 3.8]{Shimomura}}]\label{BZ}
	Let $\sigma : A \rightarrow A^+$ be a substitution. If $\sigma$ is wild, then $X_\sigma$ contains a periodic word in $B^\mathbb{Z}$.
\end{lemma}
\begin{proof}
	Let $c \in C_{liso}$ (resp. $c \in C_{riso}$), so that there exist $n \geq 1$, $u \in B^+$ and $v \in A^*$ such that $\sigma^n(c) = ucv$ (resp. $\sigma^n(c) = vcu$). Then, for all $l \geq 1$, there exists $v_l \in A^*$ such that $\sigma^{nl}(c) = \sigma^{l-1}(u)...\sigma(u) ucv_l$ (resp. $\sigma^{nl}(c) = v_lcu \sigma(u)...\sigma^{l-1}(u)$). As $u \in B^+$, the sequence $\left(\sigma^l(u)\right)_{l \geq 0}$ is eventually periodic so, as $l$ grows, the words $\sigma^{nl}(c)$ produce an infinite periodic word in $X_\sigma \cap B^\mathbb{Z}$.
\end{proof}

\begin{remark}\label{dependance}
	The proof of \Cref{BZ} is in fact more precise than its statement: it shows that every letter of $C_{liso}$ (resp. of $C_{riso}$) generates a periodic word $x \in X_\sigma \cap B^\mathbb{Z}$, therefore it generates the minimal component $X(x) \subset X_\sigma \cap B^\mathbb{Z}$.
\end{remark}

\begin{example}\label{ex2}
	Consider the substitution $\sigma : 0 \mapsto 101, 1 \mapsto 1$, for which $C = \{0\}$ and $B = \{1\}$. We have $0 \in C_{liso} \cap C_{riso}$ and ${}^\omega 1^\omega \in X_\sigma \cap B^\mathbb{Z}$ so $\{{}^\omega 1^\omega\} \subset X_\sigma \cap B^\mathbb{Z}$ is a minimal component. Note that $X_\sigma = X({}^\omega 101^\omega)$ so $\{{}^\omega 1^\omega\}$ is its unique minimal component, thus $X_\sigma$ is essentially minimal but not minimal.
\end{example}

More generally, every word in $X_\sigma \cap B^\mathbb{Z}$ has a periodic structure.

\begin{proposition}[{\cite[Proposition 4.3]{BPR2}}]\label{uvw}
	Let $\sigma : A \rightarrow A^+$ be a substitution. If $x \in X_\sigma \cap B^\mathbb{Z}$, then there exists $u, v, w \in B^+$ and $k \in \mathbb{Z}$ such that $x = T^k({}^\omega u  w v^\omega)$, where the lengths of $u, v, w$ are bounded by a computable integer depending only on $\sigma$.
\end{proposition}

We deduce a description of the minimal components of $X_\sigma$ in $B^\mathbb{Z}$.

\begin{corollary}\label{XBZ}
	Let $\sigma : A \rightarrow A^+$ be a substitution and let $X \subset B^\mathbb{Z}$ be a minimal component of $X_\sigma$, then there exists a word $u \in B^+$ such that $X = X({}^\omega u^\omega)$. In other words, every minimal component of $X_\sigma \cap B^\mathbb{Z}$ is a single periodic orbit.
\end{corollary}
\begin{proof}
	Let $x \in X$. Then, by \Cref{uvw}, there exist $u,v,w \in B^+$ and $k \in \mathbb{Z}$ such that $x = T^k({}^\omega u  w v^\omega)$. By closeness of subshifts, ${}^\omega u^\omega \in X$ so, as $X$ is minimal, $X = X({}^\omega u^\omega)$.
\end{proof}

\subsection{Results}

In order to state our results, we distinguish two types of minimal components.

\begin{definition}
	Let $\sigma : A \rightarrow A^+$ be a substitution and let $X$ be a minimal component of $X_\sigma$. If $X \subset B^\mathbb{Z}$, we say that $X$ is \textit{wild}, otherwise we say that $X$ is \textit{tame}.
\end{definition}

This definition is inspired by the two types of minimal components previously described. On the one hand, main sub-substitutions provide tame minimal components in the growing case, so our goal is to generalize them to the general case and to show that they are precisely the minimal components. On the other hand, \Cref{XBZ} states that all the wild minimal components are single periodic orbits, so our goal is to give an explicit characterization that is suited for describing the dynamics, computing, and counting the minimal components. This is the main result of this paper:

\begin{theorem}\label{main}
	Let $\sigma : A \rightarrow A^+$ be a substitution.
	
	(i) The tame minimal components of $X_\sigma$ are the $X_\tau$ where $\tau$ is a main sub-substitution of $\sigma$, i.e., a computable substitution of the form $\sigma^k|_{D \cup B}$ where $D $ is a suitable subalphabet of $C$ and $k$ is an integer that characterizes $D$.
	
	(ii) The wild minimal components of $X_\sigma$ are the $X({}^\omega LP(c)^\omega)$ where $c \in C_{liso}$ and $LP(c)$ is a computable word in $B^+$, and the $X({}^\omega RP(c)^\omega)$ where $c \in C_{riso}$ and $RP(c)$ is a computable word in $B^+$.
\end{theorem}

\begin{remark}\label{avoid}
	In the growing case, Durand's definition of main sub-substitutions and \Cref{Cortez} rely on raising the substitution to a suitable power. This has the benefit of simplifying the cyclic behaviour of subalphabets, but we do not do this here for several reasons:
	
	(i) \Cref{Xphik} also holds for admissible substitutions on $A^\mathbb{Z}$ but not for non-admissible ones: in \Cref{ex1}, we have $X_\sigma = X({}^\omega 12^\omega) \cup X({}^\omega 21^\omega)$ and $X_{\sigma^2} = X({}^\omega 21^\omega)$.
	
	(ii) It would hide the dynamical aspects we focus on in \Cref{dycomp,dcomp}.
	
	(iii) It would make computation significantly longer when we want our result to be as efficient as possible.
\end{remark}

In \Cref{tcomp} we prove \Cref{main} (i). We first introduce minimal alphabets (\Cref{minalph}), not with matrices but with an oriented graph, and we show that the restrictions to these alphabets provide tame minimal components which generalize main sub-substitutions (\Cref{mainsub}). To prove that all tame components have this form, we use the fact that the substitution $\sigma$ induces a permutation $\tilde{\sigma}$ on the subshifts of $X_\sigma$.

In \Cref{wcomp} we prove \Cref{main} (ii) by constructing the computable words $LP(c)$ and $RP(c)$ (\Cref{LP,RP}) that depend only on $\sigma$ and the letter $c$. By construction, ${}^\omega LP(c)^\omega$ and ${}^\omega RP(c)^\omega$ are the periodic words exhibited in \Cref{dependance}, and we take inspiration from a result in D0L-systems to show that they are precisely the periodic words in \Cref{XBZ}.

In \Cref{dcomp} we show that $\tilde{\sigma}$ induces a permutation on the tame (resp. wild) minimal components of $X_\sigma$, and that its action is described by the directed graphs we built to prove our theorem.

\begin{example}\label{ex0}
	The different constructions and results throughout \Cref*{tcomp,wcomp,dcomp} will be illustrated with the substitution
	\begin{equation*}
		\sigma : 0 \mapsto 141, 1 \mapsto 00, 2 \mapsto 242, 3 \mapsto 5435, 4 \mapsto 5, 5 \mapsto 6, 6 \mapsto 5,
	\end{equation*}
	for which $C = \{0,1,2,3\}$, $B = \{4,5,6\}$ and $C_{liso} = C_{riso} = \{3\}$.
\end{example}

\begin{remark}\label{reduced}
	In practice, we display minimal components in their reduced form: for every tame component $X=X_{\sigma^k|_{D \cup B}}$, there is a unique alphabet $E$ such that $D \subset E \subset D \cup B$, $\sigma^k|_E$ is defined and admissible and $X=X_{\sigma^k|_E}$; for every wild component $X=X({}^\omega u^\omega)$, there is a primitive word $v$ such that $X=X({}^\omega v^\omega)$.
\end{remark}

In the last two sections we emphasize the effectiveness of our characterization. In \Cref{numcomp} we begin by showing a known computability result for which we could not find a proper proof:
\newpage
\begin{proposition}\label{compB}
	Let $\sigma : A \rightarrow A^+$ be a substitution. Then $B$ and $C$ are computable.
\end{proposition}

As an application of our theorem, we are able to compute and count the minimal components of a given substitution subshift. If $\sigma : A \rightarrow A^+$ is a substitution, 
the number of minimal components denoted by $MC(\sigma)$.

\begin{corollary}\label{count}
	Let $\sigma : A \rightarrow A^+$ be a substitution. Then $MC(\sigma)$ is computable.
\end{corollary}

In particular, we can decide if $X_\sigma$ is essentially minimal. A Python implementation of the computation of $B$ and $C$, of the minimal components of a given substitution subshift as well as their number can be found at \url{https://codeberg.org/RaphaelHENRY/MinimalComponents.git}. According to \Cref{reduced}, we output the tame components of the form $X_{\sigma^k|_E}$ as the couple $(E,k)$ and the wild components as the primitive word $v$.

We also bound $MC(\sigma)$ by the size of the alphabet, as in \Cref{growingBound}.

\begin{corollary}\label{bound}
	Let $\sigma : A \rightarrow A^+$ be a substitution.
	
	(i) If $\lvert B \rvert = 0$, then $MC(\sigma) \leq \lvert C \rvert = \lvert A \rvert$.
	
	(ii) If $\lvert B \rvert = 1$, then $MC(\sigma) \leq \lvert C \rvert = \lvert A \rvert - 1$.
	
	(iii) If $\lvert B \rvert \geq 2$, then $MC(\sigma) \leq 2\lvert C \rvert \leq 2\lvert A \rvert -4$.
\end{corollary}

We provide examples to show that these upper bounds - in fact every number between 1 and the upper bound - can be reached.

Finally, in \Cref{scase} we detail the computation of the minimal components for all substitutions on two letters, which provides numerous examples to illustrate the tools developed in this paper.

\begin{remark}
	The non-erasing assumption on the substitution can be removed from all our results, the only difference being that computing bounded letters is more complicated.
	
	In particular, as D0L-systems are defined with morphisms instead of substitutions, the way we characterize, compute and count the wild components of substitution subshifts can be directly used to characterize, compute and count the infinite repetitions over bounded letters in a D0L-system.
\end{remark}

To conclude, we discuss in \Cref{discuss} some open questions and generalizations of our results.

\section{Tame minimal components}\label{tcomp}

\subsection{Dynamics of alphabets}

\subsubsection{Minimal alphabets}

In this section we identify the subalphabets $D \subset A$ for which there exists $k \geq 1$ such that $\sigma^k(D) \subset D^+$ in order to study the restrictions $\sigma^k|_D : D \rightarrow D^+$. Our goal is to determine when such a restriction is l-primitive and tame (so that it generates a minimal substitution subshift) so we study the action of $\sigma$ on the subalphabets of $C$ rather than on the subalphabets of $A$, and we later add the bounded letters.

\begin{definition}\label{G}
	In order to represent how $\sigma$ acts on the subalphabets of $C$, we define the directed graph $G := (V, E)$ by
	\begin{itemize}
		\item $V := \mathscr P(C) \backslash \{\emptyset\}$,
		\item $E := \left\{\left(D, \displaystyle \bigcup_{a \in D} alph_C(\sigma(a))\right) ~\middle|~ D \in V\right\} \subset V^2$.
	\end{itemize}
	If $(D,D') \in E$, we write $D \rightarrow D'$. If $k \geq 1$ and $D \underbrace{\rightarrow ... \rightarrow}_{k \textrm{ times}} D'$, we write $D \xrightarrow[k]{} D'$, and in that case $D' = \displaystyle \bigcup_{a \in D} alph_C\left(\sigma^k(a)\right)$.
\end{definition}

\begin{example}\label{exG}
	Following \Cref{ex0}, the graph $G$ has 15 vertices so let us display the orbit of the singletons only:
	
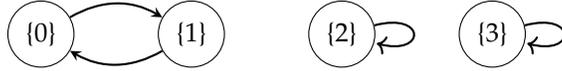
\begin{center}
    \begin{tikzpicture} [node distance = 2cm, on grid, auto]
        \node (q0) [state] {$\{0\}$};
        \node (q1) [state, right = of q0] {$\{1\}$};
        \node (q2) [state, right = of q1] {$\{2\}$};
        \node (q3) [state, right = of q2] {$\{3\}$};
        \path [-stealth, thick]
            (q0) edge [bend left] node {} (q1)
            (q1) edge [bend left] node {} (q0)
            (q2) edge [loop right] node {} (q2)
            (q3) edge [loop right] node {} (q3);
	\end{tikzpicture}
	\end{center}
\end{example}

The following lemma shows that $G$ behaves well with inclusion and union.

\begin{lemma}\label{inclusion}
	Let $k \geq 1$ and $D_1,D_2,D_3,D_4 \in V$.
	
	(i) If $D_1 \subset D_2$, $D_1 \xrightarrow[k]{} D_3$ and $D_2 \xrightarrow[k]{} D_4$, then $D_3 \subset D_4$.
	
	(ii) If $D_1 \rightarrow D_3$ and $D_2 \rightarrow D_4$, then $D_1 \cup D_2 \rightarrow D_3 \cup D_4$.
\end{lemma}
\begin{proof}
	The proof is left to the reader.
\end{proof}
\newpage
\begin{remark}\label{generators}
	\Cref{inclusion} (ii) implies that the graph $G$ is entirely determined by the alphabets $D_a$ such that $\{a\} \rightarrow D_a$ for each $a \in C$. We call these alphabets the \textit{generators} of $G$, they are particularly relevant when we do computations.
\end{remark}

New let us identify the cyclic behaviors in $G$.

\begin{definition}\label{minalph}
	We say that $D \in V$ is a \textit{$k$-periodic} alphabet if $k$ is the smallest positive integer such that $D \xrightarrow[k]{} D$. We say that $D \in V$ is a \textit{minimal} alphabet if there exists $k \geq 1$ such that $D$ is $k$-periodic and has no proper periodic subalphabet.
\end{definition}

Note that every periodic alphabet contains at least one minimal alphabet.

\begin{lemma}\label{Eminimal}
	Let $D$ be a minimal alphabet of period $k$ and let $E$ such that $D \rightarrow E$. Then $E$ is a minimal alphabet of period $k$.
\end{lemma}
\begin{proof}
	First, $E$ is a $k$-periodic alphabet. Let $E' \subset E$ be a $k'$-periodic subalphabet. Also let $D'$ be the $k'$-periodic alphabet such that $E' \xrightarrow[k-1]{} D'$. We have $E \xrightarrow[k-1]{} D$ so, by \Cref{inclusion} (i), $D' \subset D$. Then, as $D$ is minimal, we have $D'=D$ and $k'=k$. Finally, $E' \xrightarrow[k]{} E'$ and $E' \xrightarrow[k-1]{} D \rightarrow E$ so $E'=E$.
\end{proof}

We now show an equivalent property to the minimality of alphabets, which will be easier to handle in the proofs.

\begin{lemma}\label{puitminimal}
	Let $D$ be a $k$-periodic alphabet. Then $D$ is minimal if and only if for all $a \in D$, there exists $l_a \geq 1$ such that $\{a\} \xrightarrow[kl_a]{} D$.
\end{lemma}
\begin{proof}
	Suppose that $D$ is minimal. Let $a \in D$, and for all $l \geq 1$, let $D_l$ be the subalphabet such that $\{a\} \xrightarrow[kl]{} D_l$. Then, as $\{a\} \subset D$ and $D \xrightarrow[kl]{} D$, by \Cref{inclusion} (i), $D_l \subset D$. As $V$ is finite, the sequence $\left(D_l\right)_{l \geq 1}$ is eventually periodic, which means that there exists $l_a \geq 1$ such that $D_{l_a}$ is a periodic alphabet. Then, by minimality of $D$, $D_{l_a} = D$.
	
	Suppose that for all $a \in D$, there exists $l_a \geq 1$ such that $\{a\} \xrightarrow[kl_a]{} D$. Let $D' \subset D$ be a minimal $k'$-periodic alphabet and let $a \in D'$. By supposition, there exists $l_a$ such that $\{a\} \xrightarrow[kl_a]{} D$. Moreover, by the previous implication, there exists $l'_a \geq 1$ such that $\{a\} \xrightarrow[k'l'_a]{} D'$. Setting $l = max(l_a,l'_a)$, we get $\{a\} \xrightarrow[kk'l]{} D$ and $\{a\} \xrightarrow[kk'l]{} D'$, so $D=D'$. Hence $D$ is minimal.
\end{proof}

\begin{remark}
	\Cref{puitminimal} implies that the minimal alphabets are in the orbit of the singletons $\{a\}$ for $a \in C$. Therefore, when searching for minimal alphabets we only need to compute these orbits instead of the entire graph $G$.
\end{remark} 

\begin{example}\label{exminalph}
	Following \Cref{exG}, the minimal alphabets are the 2-periodic alphabets $\{0\}$ and $\{1\}$ and the 1-periodic alphabets $\{2\}$ and $\{3\}$.
\end{example}

We also prove a natural fact that will be useful when we count minimal components in \Cref{numcomp}.

\begin{proposition}\label{compdisjoint}
	The minimal alphabets are pairwise disjoint.
\end{proposition}
\begin{proof}
	Let $D$ and $D'$ be two minimal alphabets of respective period $k$ and $k'$ such that $D \cap D' \neq \emptyset$. Let $a \in D \cap D'$. By \Cref{puitminimal}, there exists $l_a \geq 1$ such that $\{a\} \xrightarrow[kl_a]{} D$ and $l'_a \geq 1$ such that $\{a\} \xrightarrow[k'l'_a]{} D'$. By setting $l = \max\left(l_a, l'_a\right)$, we get $\{a\} \xrightarrow[kk'l]{} D$ and $\{a\} \xrightarrow[kk'l]{} D'$, hence $D=D'$.
\end{proof}

\subsubsection{l-primitivity on subalphabets}

If $D$ is a $k$-periodic alphabet, we have the restriction $\sigma^k|_{D \cup B} : D \cup B \rightarrow (D \cup B)^+$. We show here that the restrictions to minimal alphabets are precisely the l-primitive sub-substitutions of $\sigma$.

\begin{proposition}\label{minimalprimitive}
	Let $D$ be a $k$-periodic alphabet. Then, $D$ is minimal if and only if $\sigma^k|_{D \cup B}$ is l-primitive.
\end{proposition}
\begin{proof}
	Thanks to \Cref{puitminimal}, we are going to show that $\sigma^k|_{D \cup B}$ is l-primitive if and only if for all $a \in D$, there exists $l_a \geq 1$ such that $\{a\} \xrightarrow[kl_a]{} D$.

	Suppose that $\sigma^k|_{D \cup B}$ is l-primitive, which provides $n \geq 1$ such that, for all $a,b \in D$, $a \sqsubset \sigma^{kn}(b)$. Then, for all $a \in D$, we have $alph_C\left(\sigma^{kn}(a)\right) = D$, which means that $\{a\} \xrightarrow[kn]{} D$.
	
	Suppose that, for all $a \in D$, there exists $l_a \geq 1$ such that $\{a\} \xrightarrow[kl_a]{} D$. By setting $n = \displaystyle \max_{a \in D} l_a$, we get $\{a\} \xrightarrow[kn]{} D$ for all $a \in D$, i.e., $alph_C(\sigma^{kn}(a)) = D$, which means that $\sigma^k|_{D \cup B}$ is l-primitive.
\end{proof}

\subsubsection{Dynamics of subshifts}\label{dycomp}

A substitution $\sigma : A \rightarrow A^+$ induces a map $\tilde{\sigma}$ on the subshifts of $X_\sigma$: if $X \subset X_\sigma$ is a subshift, we define the subshift
\begin{equation*}
	\tilde{\sigma}(X) := \left\{T^k\left(\sigma(x)\right) ~\middle|~ x \in X, k \in \mathbb{Z}\right\} \subset X_\sigma.
\end{equation*}

We recall that every word in a substitution subshift can be desubstituted. A proof can be found in \cite[Theorem 5.1]{BPR} for example:

\begin{lemma}\label{phiminus}
	For all $x \in X_\sigma$, there exist $y \in X_\sigma$ and $k \in \mathbb{Z}$ such that $x=T^k(\sigma(y))$.
\end{lemma}

This means that the reciprocal of $\tilde{\sigma}$ is defined, it is the subshift
\begin{equation*}
	\tilde{\sigma}^{-1}(X) := \left\{x \in X_\sigma ~\middle|~ \sigma(x) \in X\right\} \subset X_\sigma,
\end{equation*}
such that, for all subshift $X \subset X_\sigma$, $\tilde{\sigma}\left(\tilde{\sigma}^{-1}(X)\right) = X$. More generally, for all subshift $X \subset X_\sigma$ and all $n \geq 1$, $\tilde{\sigma}^n\left(\tilde{\sigma}^{-n}(X)\right) = X$.

\begin{remark}\label{phin}
	If $X \subset X_\sigma$ is a subshift such that $alph_C(X) \neq \emptyset$, then $alph_C\left(\tilde{\sigma}(X)\right) \neq \emptyset$ and $alph_C(X) \rightarrow alph_C\left(\tilde{\sigma}(X)\right)$. More generally, for $n \geq 1$, $alph_C(X) \xrightarrow[n]{} alph_C\left(\tilde{\sigma}^n(X)\right)$.
\end{remark}

We can now prove that the subshifts of $X_\sigma$ have a strong cyclic behavior.

\begin{proposition}\label{Xperiodic}
	Let $X \subset X_\sigma$ be a subshift. Then there exists $k \geq 1$ such that $\tilde{\sigma}^k(X) = X$, and if in addition $alph_C(X) \neq \emptyset$, then $alph_C(X)$ is a $k$-periodic alphabet and $X_{\sigma^k|_{alph_C(X) \cup B}} \subset X$.
\end{proposition}
\begin{proof}
	Consider the sequence of subshifts $\left(\tilde{\sigma}^{-n}(X)\right)_{n \geq 0}$. By \Cref{quasiminimal}, $X_\sigma$ has a finite number of subshifts so this sequence is ultimately periodic. By setting $l \geq 0$ and $k \geq 1$ the smallest integers such that $\tilde{\sigma}^{-l}(X)=\tilde{\sigma}^{-(l+k)}(X)$, we obtain $\tilde{\sigma}^k(X) = \tilde{\sigma}^{l+k}(\tilde{\sigma}^{-l}(X)) = \tilde{\sigma}^{l+k}(\tilde{\sigma}^{-(l+k)}(X)) = X$.
	
	If $alph_C(X) \neq \emptyset$, with \Cref{phin} we have $alph_C(X) \xrightarrow[k]{} alph_C(\tilde{\sigma}^k(X)) = alph_C(X)$, which means that $alph_C(X)$ is $k$-periodic. Let $u \in \mathscr L\left(X_{\sigma^k|_{alph_C(X) \cup B}}\right)$, then there exist $c \in alph_C(X)$ and $l \geq 1$ such that $u \sqsubset \sigma^{kl}(c)$. As $c \in alph_C(X)$, there exists $x \in X$ such that $c \sqsubset x$, then $u \sqsubset \sigma^{kl}(c) \sqsubset \sigma^{kl}(x) \in \tilde{\sigma}^{kl}(X) = X$ so $u \in \mathscr L(X)$. Therefore $\mathscr L\left(X_{\sigma^k|_{alph_C(X) \cup B}}\right) \subset \mathscr L(X)$, which means that $X_{\sigma^k|_{alph_C(X) \cup B}} \subset X$.
\end{proof}

\begin{remark}\label{permutation}
	This means that $\tilde{\sigma}$ is a permutation on the subshifts of $X_\sigma$.
\end{remark}

\subsection{Proof of \Cref*{main} (i)}

We are now able to generalize main sub-substitutions to the general case.

\begin{definition}\label{mainsub}
	Let $\sigma : A \rightarrow A^+$ be a substitution. If $D \subset C_{niso}$ is a minimal alphabet of period $k$, we say that the substitution $\sigma^k|_{D \cup B}$ is a \textit{main sub-substitution} of $\sigma$. In particular, main sub-substitutions are tame (because $D \subset C_{niso}$) and l-primitive (by \Cref{minimalprimitive}).
\end{definition}

\begin{example}\label{exmainsub}
	Following \Cref{exminalph}, we recall that $C_{niso} = \{0,1,2\}$ so the minimal alphabets included in $C_{niso}$ are $\{0\}$ of period 2, $\{1\}$ of period 2 and $\{2\}$ of period 1. We then add the bounded letters that appear in the associated subshift, so the main sub-substitutions of $\sigma$ are $\sigma^2|_{\{0,5\}}$, $\sigma^2|_{\{1,4,6\}}$ and $\sigma|_{\{2,4,5,6\}}$.
\end{example}

\Cref{Shimomura} (i) directly provides the following result.

\begin{proposition}\label{tame1}
	Let $\tau$ be a main-substitution of $\sigma$. Then $X_\tau$ is a minimal component of $X_\sigma$.
\end{proposition}

The converse relies on \Cref{Xperiodic}.

\begin{proposition}\label{tame2}
	Let $X$ be a tame minimal component of $X_\sigma$. Then there exists a main sub-substitution $\tau$ of $\sigma$ such that $X = X_\tau$.
\end{proposition}
\begin{proof}
	First, \Cref{Xperiodic} provides $k \geq 1$ such that $alph_C(X)$ is $k$-periodic and $X_{\sigma^k|_{alph_C(X) \cup B}} \subset X$. Then, by minimality of $X$, we have equality. It remains to show that $\sigma^k|_{alph_C(X) \cup B}$ is a main sub-substitution of $\sigma$.
	
	As $alph_C(X) \neq \emptyset$ and $X$ is minimal, growing letters occur with bounded gaps in $X$. In other words, the words in $\mathscr L\left(X_{\sigma^k|_{alph_C(X) \cup B}}\right) \cap B^*$ have bounded length, so, by \Cref{Lbounded}, $\sigma^k|_{alph_C(X) \cup B}$ is tame, i.e., $alph_C(X) \subset C_{niso}$.
	
	\Cref{Xperiodic} also states that $\tilde{\sigma}^k(X) = X$, so, for all $d \in alph_C(X)$, $ \sigma^{kl}(d) \in \mathscr L(X)$. As $X$ is minimal, for $l$ large enough, $\sigma^{kl}(d)$ contains every growing letter of $alph_C(X)$. This means that $\sigma^k|_{alph_C(X) \cup B}$ is l-primitive, and by \Cref{minimalprimitive} $alph_C(X)$ is minimal. Therefore, $\sigma^k|_{alph_C(X) \cup B}$ is a main sub-substitution.
\end{proof}

\Cref{tame1,tame2} complete the proof of \Cref{main} (i).

\begin{example}\label{extcomp}
	Following \Cref{exmainsub}, the tame minimal components of $X_\sigma$ are $X_{\sigma^2|_{\{0,5\}}}$, $X_{\sigma^2|_{\{1,4,6\}}}$ and $X_{\sigma|_{\{2,4,5,6\}}}$.
\end{example}

\section{Wild minimal components}\label{wcomp}

\subsection{Maximal bounded factors}

The purpose of this section is to prove \Cref{LXB}, which is an analog of a result proved for D0L-systems:

\begin{proposition}[{\cite[Theorem 12]{KS}}]\label{3forms}
	Let $G = (A, \sigma, w)$ be a pushy D0L-system (i.e., $\sigma$ is wild). Then there exists $L \geq 1$ and a finite subset $U \subset B^+$ such that any factor from $\mathscr L(G) \cap B^*$ has one of the following forms:
	
	• $w_1$
	
	• $w_1 u_1^{k_1} w_2$
	
	• $w_1 u_1^{k_1} w_2 u_2^{k_2} w_3$
	
	where $u_1, u_2 \in U$ , $\lvert w_j \rvert < L$ for all $j \in \{1, 2, 3\}$, and $k_1, k_2 \geq 1$.
\end{proposition}

A similar result can also be found for purely morphic words in \cite[Proposition 4.7.62]{CANT}. We are going to be more precise by explicitly defining the words $u_1$ and $u_2$ as the computable words $LP(c)$ and $RP(c)$ for specific growing letters $c$. To achieve that, we define the following.

\begin{definition}
	If $c \in C$ and $k \geq 1$, we say that $v$ is a \textit{maximal bounded factor of $\sigma^k(c)$} if $v \in \mathscr L\left(\sigma^k(c)\right) \cap B^*$ and it is not a factor of any other word of $\mathscr L\left(\sigma^k(c)\right) \cap B^*$. We also say that a word is a \textit{maximal bounded factor of $\sigma$} if it is a maximal bounded factor of a $\sigma^k(c)$ for some $c \in C$ and $k \geq 0$.
\end{definition}

The purpose of maximal bounded factors is that they contain the structure of every word of $\mathscr L(X_\sigma) \cap B^*$, as highlighted by the following remark.

\begin{remark}
	As only growing letters generate images of arbitrarily large length, for all $u \in \mathscr L(X_\sigma) \cap B^*$, there exists $c \in C$ and $k \geq 0$ such that $u \sqsubset \sigma^k(c)$. In particular, $u$ is a factor of a maximal bounded factor of $\sigma^k(c)$.
\end{remark}

\subsubsection{1-blocks}

We adapt here the notion of 1-blocks, introduced by Devyatov in \cite{Devyatov} for morphic words, to substitution subshifts.

\begin{definition}
	If $u \in A^*$ contains a growing letter, we denote the first growing letter in $u$ by $LC(u)$ and the prefix before $LC(u)$ by $LB(u) \in B^*$. Symmetrically, we denote the last growing letter in $u$ by $RC(u)$ and the suffix after $RC(u)$ by $RB(u) \in B^*$.
\end{definition}

\newcommand{\exo}{
   \begin{gathered}
    \tikzpicture[every node/.style={anchor=south west}]
        \node[minimum width=0.5cm,minimum height=0.5cm] at (0.1,0) {22};
        \node[minimum width=0.5cm,minimum height=0.5cm] at (0.75,0) {0};
        \node[minimum width=0.5cm,minimum height=0.5cm] at (1.25,0) {1200};
        \node[minimum width=0.5cm,minimum height=0.5cm] at (2.25,0) {1};
        \node[minimum width=0.5cm,minimum height=0.5cm] at (2.75,0) {$\varepsilon$};
        \draw (0,0) -- (3.25,0);
        \draw (0,0) -- (0,0.5);
        \draw (0.75,0) -- (0.75,0.5);
        \draw (1.25,0) -- (1.25,0.5);
        \draw (2.25,0) -- (2.25,0.5);
        \draw (2.75,0) -- (2.75,0.5);
        \draw (3.25,0) -- (3.25,0.5);
        \draw (0,0.5) -- (3.25,0.5);
        \draw [decorate, decoration = {brace}] (0.75,0) --  (0,0)
			node[pos=0.5,below=2pt,black]{\scriptsize LB};
		\draw [decorate, decoration = {brace}] (1.25,0) --  (0.75,0)
			node[pos=0.5,below=2pt,black]{\scriptsize LC};
		\draw [decorate, decoration = {brace}] (2.75,0) --  (2.25,0)
			node[pos=0.5,below=2pt,black]{\scriptsize RC};
			\draw [decorate, decoration = {brace}] (3.25,0) --  (2.75,0)
			node[pos=0.5,below=2pt,black]{\scriptsize RB};
    \endtikzpicture
    \end{gathered}
}

\newcommand{\ext}{
   \begin{gathered}
    \tikzpicture[every node/.style={anchor=south west}]
        \node[minimum width=0.5cm,minimum height=0.5cm] at (0,0) {2};
        \node[minimum width=0.5cm,minimum height=0.5cm] at (1,0) {0};
        \node[minimum width=0.5cm,minimum height=0.5cm] at (2,0) {2};
        \draw (0,0) -- (2.5,0);
        \draw (0,0) -- (0,0.5);
        \draw (0.5,0) -- (0.5,0.5);
        \draw (2,0) -- (2,0.5);
        \draw (2.5,0) -- (2.5,0.5);
        \draw (0,0.5) -- (2.5,0.5);
        \draw [decorate, decoration = {brace}] (0.5,0) --  (0,0)
			node[pos=0.5,below=2pt,black]{\scriptsize LB};
		\draw [decorate, decoration = {brace}] (2,0) --  (0.5,0)
			node[pos=0.5,below=2pt,black]{\scriptsize LC=RC};
		\draw [decorate, decoration = {brace}] (2.5,0) --  (2,0)
			node[pos=0.5,below=2pt,black][fontscale=0.1]{\scriptsize RB};
    \endtikzpicture
    \end{gathered}
}

\begin{example}
	If $C=\{0,1\}$ and $B = \{2\}$, the words 22012001 and 202 have the following decomposition:
	\begin{center}
		$\exo$ and $\ext$.
	\end{center}
\end{example}

\begin{definition}
	If $c \in C$ and $k \geq 1$, a \textit{1-block of $\sigma^k(c)$} is a triplet $(a, u ,b) \in C \times B^* \times C$ such that $aub \sqsubset \sigma^k(c)$. In general, we say that $(a, u ,b)$ is a \textit{1-block of $\sigma$} if it is a 1-block of a $\sigma^k(c)$ for some $c \in C$ and $k \geq 1$.
\end{definition}

	 We now note an important fact.

\begin{remark}\label{maxf}
	Let $c \in C$ and $k \geq 1$. The maximal bounded factors of $\sigma^k(c)$ have one of the following forms:
	
	• $LB \left (\sigma^k(c) \right)$,
	
	• $u$ where $(a,u,b)$ is a 1-block of $\sigma^k(c)$,
	
	• $RB \left (\sigma^k(c) \right )$.
\end{remark}

Note that a 1-block of $\sigma$ can be a 1-block of several $\sigma^k(c)$, but we need to consider a special case.

\begin{definition}
	A 1-block of $\sigma(c)$ for some $c \in C$ is called an \textit{origin of $\sigma$}. In particular, $\sigma$ has a finite number of origins.
\end{definition}

We now show that every 1-block of $\sigma$ can be recovered from an origin of $\sigma$.

\begin{definition}
	Let $\mathfrak{u} = (a, u, b)$ be a 1-block of $\sigma^k(c)$ for some $c \in C$ and $k \geq 1$. As $aub \sqsubset \sigma^k(c)$, we have $\sigma(a)~\sigma(u)~\sigma(b)~\sqsubset \sigma^{k+1}(c)$, and in particular
	\begin{equation*}
		\mathfrak{v} = \left(~RC\left(\sigma(a)\right)~,~RB\left(\sigma(a)\right)~\sigma(u)~LB\left(\sigma(b)\right)~,~LC\left(\sigma(b)\right)~\right)
	\end{equation*}
	is a 1-block of $\sigma^{k+1}(c)$. We call $\mathfrak{v}$ the \textit{descendant} of $\mathfrak{u}$ and we write $\mathfrak{v} = \mathfrak{D}(\mathfrak{u})$. For $l \geq 0$, we also write $\mathfrak{D}^l$ the $l$-th iteration of $\mathfrak{D}$, in particular $\mathfrak{D}^0(\mathfrak{u}) = \mathfrak{u}$.
\end{definition}

\newcommand{\Co}{
   \begin{gathered}
    \tikzpicture[every node/.style={anchor=south west}]
        \node[minimum width=0.5cm,minimum height=0.5cm] at (1.875,0) {$...$};
        \node[minimum width=0.5cm,minimum height=0.5cm] at (2.75,0) {$a$};
        \node[minimum width=0.5cm,minimum height=0.5cm] at (3.5,0) {$u$};
        \node[minimum width=0.5cm,minimum height=0.5cm] at (4.25,0) {$b$};
        \node[minimum width=0.5cm,minimum height=0.5cm] at (5.125,0) {$...$};
    	\node[minimum width=0.5cm,minimum height=0.5cm] at (2.75,1.2) {$...$};
        \node[minimum width=0.5cm,minimum height=0.5cm] at (3.5,1.2) {$a'$};
        \node[minimum width=0.5cm,minimum height=0.5cm] at (4.25,1.2) {$...$};
        \draw (1.5,0) -- (6,0);
        \draw (1.5,0) -- (1.5,0.5);
        \draw [decorate, decoration = {brace}] (1.5,0) --  (1.5,0.5)
			node[pos=0.5,left=2pt,black]{$\sigma^{k+1}(c)$};
        \draw (2.75,0) -- (2.75,0.5);
        \draw (3.25,0) -- (3.25,0.5);
        \draw (4.25,0) -- (4.25,0.5);
        \draw (4.75,0) -- (4.75,0.5);
        \draw (6,0) -- (6,0.5);
        \draw (1.5,0.5) -- (6,0.5);
        \draw[->] (2.5,1.2) -- (1.5,0.5);
        \draw[->] (3.5,1.2) -- (2.5,0.5);
        \draw[->] (4,1.2) -- (5,0.5);
        \draw[->] (5,1.2) -- (6,0.5);
        \draw (2.5,1.2) -- (5,1.2);
        \draw (2.5,1.2) -- (2.5,1.7);
        \draw [decorate, decoration = {brace}] (2.5,1.2) --  (2.5,1.7)
			node[pos=0.5,left=2pt,black]{$\sigma^k(c)$};
        \draw (3.5,1.2) -- (3.5,1.7);
        \draw (4,1.2) -- (4,1.7);
        \draw (5,1.2) -- (5,1.7);
        \draw (2.5,1.7) -- (5,1.7);
    \endtikzpicture
    \end{gathered}
}

\newcommand{\Ct}{
   \begin{gathered}
    \tikzpicture[every node/.style={anchor=south west}]
        \node[minimum width=0.5cm,minimum height=0.5cm] at (1.875,0) {$...$};
        \node[minimum width=0.5cm,minimum height=0.5cm] at (2.75,0) {$a$};
        \node[minimum width=0.5cm,minimum height=0.5cm] at (3.75,0) {$u$};
        \node[minimum width=0.5cm,minimum height=0.5cm] at (4.75,0) {$b$};
        \node[minimum width=0.5cm,minimum height=0.5cm] at (5.625,0) {$...$};
    	\node[minimum width=0.5cm,minimum height=0.5cm] at (2.25,1.2) {$...$};
        \node[minimum width=0.5cm,minimum height=0.5cm] at (3,1.2) {$a'$};
        \node[minimum width=0.5cm,minimum height=0.5cm] at (3.75,1.2) {$u'$};
        \node[minimum width=0.5cm,minimum height=0.5cm] at (4.5,1.2) {$b'$};
        \node[minimum width=0.5cm,minimum height=0.5cm] at (5.25,1.2) {$...$};
        \draw (1.5,0) -- (6.5,0);
        \draw (1.5,0) -- (1.5,0.5);
        \draw [decorate, decoration = {brace}] (1.5,0) --  (1.5,0.5)
			node[pos=0.5,left=2pt,black]{$\sigma^{k+1}(c)$};
        \draw (2.75,0) -- (2.75,0.5);
        \draw (3.25,0) -- (3.25,0.5);
        \draw (4.75,0) -- (4.75,0.5);
        \draw (5.25,0) -- (5.25,0.5);
        \draw (6.5,0) -- (6.5,0.5);
        \draw (1.5,0.5) -- (6.5,0.5);
        \draw[->] (2,1.2) -- (1.5,0.5);
        \draw[->] (3,1.2) -- (2.5,0.5);
        \draw[->] (3.5,1.2) -- (3.5,0.5);
        \draw[->] (4.5,1.2) -- (4.5,0.5);
        \draw[->] (5,1.2) -- (5.5,0.5);
        \draw[->] (6,1.2) -- (6.5,0.5);
        \draw (2,1.2) -- (6,1.2);
        \draw (2,1.2) -- (2,1.7);
        \draw [decorate, decoration = {brace}] (2,1.2) --  (2,1.7)
			node[pos=0.5,left=2pt,black]{$\sigma^k(c)$};
        \draw (3,1.2) -- (3,1.7);
        \draw (3.5,1.2) -- (3.5,1.7);
        \draw (4.5,1.2) -- (4.5,1.7);
        \draw (5,1.2) -- (5,1.7);
        \draw (6,1.2) -- (6,1.7);
        \draw (2,1.7) -- (6,1.7);
    \endtikzpicture
    \end{gathered}
}

\begin{lemma}\label{origin}
	Let $c \in C$. For all $k \geq 1$ and all 1-block $\mathfrak{u}$ of $\sigma^k(c)$, there exist an origin $\mathfrak{v}$ of $\sigma$ and $0 \leq l < k$ such that $\mathfrak{u} = \mathfrak{D}^l(\mathfrak{v})$.
\end{lemma}
\begin{proof}
	Let us prove it by induction on $k$. If $k=1$, $\mathfrak{u}$ is itself an origin of $\sigma$ and $\mathfrak{u} = \mathfrak{D}^0(\mathfrak{u})$. Let $k \geq 1$ be such that, for every $c \in C$ and every 1-block $\mathfrak{u}$ of $\sigma^k(c)$, there exist an origin $\mathfrak{v}$ of $\sigma$ and $0 \leq l < k$ such that $\mathfrak{u} = \mathfrak{D}^l(\mathfrak{v})$. Let $c \in C$ and let $\mathfrak{u} = (a,u,b)$ be a 1-block of $\sigma^{k+1}(c)$. When desubstituting once, the growing letters $a$ and $b$ come from growing letters in $\sigma^k(c)$, and we have two possibilities:
	 
	• Either $a$ and $b$ come from the same growing letter $a' \sqsubset \sigma^k(c)$, which means that $a u b \sqsubset \sigma(a') \sqsubset \sigma^{k+1}(c)$. This is illustrated by \Cref*{fig1}:
	
	\begin{figure}[H]
		\begin{equation*}
			\Co
		\end{equation*}
		\caption{$a u b \sqsubset \sigma(a') \sqsubset \sigma^{k+1}(c)$.}
		\label{fig1}
	\end{figure}
	
	Then $aub \sqsubset \sigma(a')$ means that $\mathfrak{u}$ is itself an origin.
	
	• Or $a$ and $b$ come from two different growing letters $a', b' \sqsubset \sigma^k(c)$, which means that there exists $u' \in A^*$ such that $a' u' b' \sqsubset \sigma^k(c)$ and $aub \sqsubset \sigma^{k+1}(c)$. This is illustrated by \Cref*{fig2}:
	\begin{figure}[H]
		\begin{equation*}
			\Ct
		\end{equation*}
		\caption{$aub \sqsubset \sigma(a'u'b') \sqsubset \sigma^{k+1}(c)$.}
		\label{fig2}
	\end{figure}
	
	As $\sigma(u') \sqsubset u \in B^*$, we have $u' \in B^*$. In particular, $\mathfrak{u'} = (a',u',b')$ is a 1-block of $\sigma^k(c)$ so, by hypothesis, there exists an origin $\mathfrak{v}$ and $0 \leq l < k$ such that $\mathfrak{u'} = \mathfrak{D}^l(\mathfrak{v})$. We have $a = RC(\sigma(a'))$ and $b = LC(\sigma(b'))$ so $u = RB(\sigma(a'))~\sigma(u')~LB(\sigma(b'))$, therefore $\mathfrak{u} = \mathfrak{D}(\mathfrak{u'}) = \mathfrak{D}^{l+1}(\mathfrak{v})$.
\end{proof}

This leads to the following definition.

\begin{definition}
	An \textit{evolution of 1-blocks of $\sigma$} is a sequence $\mathscr E = \left( \mathscr E_l \right)_{l \geq 0}$ where $\mathscr E_0$ is an origin of $\sigma$ and, for all $l \geq 1$, $\mathscr E_l = \mathfrak{D}^l(\mathscr E_0)$. In particular, each evolution is determined by its origin, $\sigma$ has a finite number of evolutions of 1-blocks and \Cref{origin} means that every 1-block of $\sigma$ belongs to an evolution.
\end{definition}

\subsubsection{Decomposition of 1-blocks}

In order to study the structure of 1-blocks, let us start with a useful lemma.

\begin{lemma}\label{LBRBphi}
	Let $u$ be a word containing a growing letter. Then
	\begin{align*}
		LB\left(\sigma(u)\right) = &~\sigma(LB(u))~LB \left(\sigma(LC(u))\right), \\
		RB\left(\sigma(u)\right) = &~RB \left(\sigma(RC(u))\right)~\sigma(RB(u)).
	\end{align*}
\end{lemma}
\begin{proof}
	We have $LB\left(\sigma(u)\right) = LB \left(\sigma(LB(u)) ~ \sigma (LC(u))...\right) = \sigma(LB(u)) ~ LB \left(\sigma(LC(u))\right)$. The right case is symmetric.
\end{proof}

We deduce a decomposition of 1-blocks.

\begin{lemma}\label{El}
	Let $\mathscr E$ be an evolution of 1-blocks of $\sigma$ and set $\mathscr E_0 = (a,u,b)$. For all $l \geq 0$, we have
    	\begin{equation*}
        	\mathscr E_l = \left(RC(\sigma^l(a))~,~RB(\sigma^l(a)) ~ \sigma^l(u) ~LB(\sigma^l(b))~,~LC(\sigma^l(b)) \right).
    	\end{equation*}
\end{lemma}
\begin{proof}
	Let us prove it by induction on $l$. If $l = 0$, this is directly true. Let $l \geq 0$ such that $\mathscr E_l = \left(RC(\sigma^l(a))~,~RB(\sigma^l(a)) ~ \sigma^l(u) ~LB(\sigma^l(b))~,~LC(\sigma^l(b)) \right)$. If $\mathscr E_{l+1} = (a', u', b')$, we have $\mathscr E_{l+1} = \mathfrak{D}(\mathscr E_l)$ so $a' = RC(\sigma(RC(\sigma^l(a)))) = RC(\sigma^{l+1}(a))$,
	\begin{align*}
		u' = & ~ RB(\sigma(RC(\sigma^l(a)))) ~ \sigma(RB(\sigma^l(a))) ~ \sigma^{l+1}(u) ~ \sigma(LB(\sigma^l(b))) ~ LB(\sigma(LC(\sigma^l(b)))) \\
		= & ~ RB(\sigma^{l+1}(a)) ~ \sigma^{l+1}(u) ~ LB(\sigma^{l+1}(b)) \textrm{ with \Cref{LBRBphi}},
	\end{align*}
	
	and $b' = LC(\sigma(LC(\sigma^l(b)))) = LC(\sigma^{l+1}(b))$.
\end{proof}

This lemma allows us to refine \Cref{maxf}:

\begin{remark}\label{maxfac}
	Let $c \in C$ and $k \geq 1$. The maximal bounded factors of $\sigma^k(c)$ have one of the following forms:
	
	(i) $LB\left(\sigma^k(c)\right)$,
	
	(ii) $RB\left(\sigma^l(a)\right)~\sigma^l(u)~LB\left(\sigma^l(b)\right)$ where $0 \leq l<k$ and $(a,u,b)$ is an origin of $\sigma$,
	
	(iii) $RB\left(\sigma^k(c)\right)$.
\end{remark}

Now let us decompose the words $LB\left(\sigma^k(c)\right)$ and $RB \left(\sigma^k(c)\right)$ for $c \in C$ and $k \geq 0$.

\subsubsection{Decomposition of $LB\left(\sigma^k(c)\right)$}

Decomposing the words $LB\left(\sigma^k(c)\right)$ is strongly related to the prefix automaton of the substitution, but we take here a different approach. To begin with, the following lemma provides a decomposition of $LB \left(\sigma^k(c)\right)$ into $k$ parts, where we heavily use the notation introduced in \Cref{concat}.

\begin{lemma}\label{LBphik}
	Let $c \in C$. Then, for all $k \geq 1$,
	\begin{equation*}
		LB\left(\sigma^k(c)\right) = \displaystyle \bigsqcup_{j=0}^{k-1} \sigma^{k-1-j}\left(LB(\sigma(LC(\sigma^j(c))))\right).
	\end{equation*}
\end{lemma}
\begin{proof}
	Let us prove it by induction on $k$. If $k = 1$, this is directly true.
	
	Let $k \geq 1$ be such that $LB \left(\sigma^k(c)\right) = \displaystyle \bigsqcup_{j=0}^{k-1} \sigma^{k-1-j}\left(LB(\sigma(LC(\sigma^j(c))))\right)$. Then 
	\begin{align*}
		LB \left(\sigma^{k+1}(c) \right) = & ~ \sigma\left(LB(\sigma^k(c))\right) ~ LB\left(\sigma(LC(\sigma^k(c)))\right) \textrm{ with \Cref{LBRBphi}} \\
		= & ~ \displaystyle \bigsqcup_{j=0}^{k-1} \sigma^{k-j}\left(LB(\sigma(LC(\sigma^j(c))))\right) ~.~ LB\left(\sigma(LC(\sigma^k(c)))\right) \\
		= & ~ \displaystyle \bigsqcup_{j=0}^k \sigma^{k-j}\left(LB(\sigma(LC(\sigma^j(c))))\right).
	\end{align*}
\end{proof}

Now we are going to describe the periodic structure of this decomposition. First, let us introduce the directed graph $G_L$, which corresponds to the directed graph $UL_G$ in \cite[Section 3.3]{KS}.

\begin{definition}\label{GL}
	In order to represent how $\sigma$ acts on the leftmost letter of $\sigma^k(c)$ for $c \in C$, we define the directed graph $G_L := \left(V_L, E_L\right)$ by
	
	• $V_L = C$,
	
	• $E_L := \left\{\left(c, LC(\sigma(c))\right) ~\middle|~ c \in C\right\}$.
	
	If $(a,b) \in E_L$, we write $a \xrightarrow[]{L} b$. Note that every vertex of $G_L$ has out degree 1.
\end{definition}

\begin{definition}
	We say that a collection $\mathscr C = (c_i)_{0 \leq i \leq p-1} \in C^p$ is a \textit{$p$-cycle of $G_L$} if the $c_i$ are all distinct and, for all $i \in \llbracket 0, p-1 \rrbracket$, $c_i \xrightarrow[]{L} c_{i+1[p]}$ where $i+1[p]$ denotes $i+1~mod~p$. If $c \in C$ belongs to a $p$-cycle of $G_L$, we say that $c$ is \textit{left-$p$-periodic}.
\end{definition}

\begin{example}\label{exGL}
	Following \Cref{ex0}, the graph $G_L$ is the following:
	\begin{center}
    	\begin{tikzpicture} [node distance = 2cm, on grid, auto]
        	\node (q0) [state] {0};
        	\node (q1) [state, right = of q0] {1};
        	\node (q2) [state, right = of q1] {2};
        	\node (q3) [state, right = of q2] {3};
        	\path [-stealth, thick]
            (q0) edge [bend left] node {L} (q1)
            (q1) edge [bend left] node {L} (q0)
            (q2) edge [loop right]node {L} (q2)
            (q3) edge [loop right]node {L} (q3);
		\end{tikzpicture}
	\end{center}
	The letters 0 and 1 are left-2-periodic, and the letters 2 and 3 are left-1-periodic.
\end{example}

\begin{definition}\label{Lc}
	If $\mathscr C = (c_i)_{0 \leq i \leq p-1}$ is a $p$-cycle of $G_L$, for all $i \in \llbracket 0,p-1 \rrbracket$, we define the word
	\begin{equation*}
		L(c_i) := \displaystyle \bigsqcup_{j=0}^{p-1} \sigma^{p-1-j}\left(LB(\sigma(c_{i+j[p]}))\right).
	\end{equation*}
\end{definition}

\begin{remark}\label{isocycle}
	Let $\mathscr C$ be a $p$-cycle of $G_L$. For all $c \in \mathscr C$, \Cref{LBphik} provides $L(c) = LB\left(\sigma^p(c)\right)$. As $p$ is the first integer such that $LC(\sigma^p(c)) = c$, this means that $L(c)$ is the smallest bounded word generated to the left of $c$: there exists $v \in A^*$ such that $\sigma^p(c) = L(c)~cv$. In particular, $c \in C_{liso}$ if and only if $L(c) \neq \varepsilon$.
	
	Also, by construction, $L(c) \neq \varepsilon$ if and only if there exists $c' \in \mathscr C$ such that $LB(\sigma(c')) \neq \varepsilon$. This means that, in a cycle of $G_L$, either every letter is left-isolated or no letter is left-isolated.
\end{remark}

\begin{definition}
	Let $\mathscr C = \{c_i\}_{0 \leq i \leq p-1}$ be a $p$-cycle of $G_L$. For all $i \in \llbracket 0,p-1 \rrbracket$, $L(c_i) \in B^*$ so the sequence $\left( \sigma^{pj} (L(c_i)) \right)_{j \geq 0}$ is eventually periodic. Set $q_{\mathscr C} \geq 0$ and $p_{\mathscr C} \geq 1$ the first integers such that, for all $i \in \llbracket 0,p-1 \rrbracket$, $\sigma^{p(q_{\mathscr C}+p_{\mathscr C})}\left(L(c_i)\right) = $ $\sigma^{pq_{\mathscr C}}\left(L(c_i)\right)$. We then define the part before the left period
	\begin{equation*}
		LQ(c) := \displaystyle \bigsqcup_{j=0}^{q_{\mathscr C}-1} \sigma^{p(q_{\mathscr C}-1-j)}(L(c)),
	\end{equation*}
	and the left period
	\begin{equation}\label{LP}
		LP(c) := \displaystyle \bigsqcup_{j=0}^{p_{\mathscr C}-1} \sigma^{p(q_{\mathscr C}+p_{\mathscr C}-1-j)}(L(c)).
	\end{equation}
	In particular, $c \in C_{liso}$ if and only if $LP(c) \neq \varepsilon$.
\end{definition}

\begin{example}\label{exLP}
	Following \Cref{exGL}, we compute $LP$ of the left-periodic letters.
	
	• In the 2-cycle of $G_L$ $\{0,1\}$, $\sigma\left(LB(\sigma(0))\right) = \sigma\left(LB(\sigma(1))\right) = \varepsilon$ so, by \Cref{isocycle}, we have $L(0) = L(1) = \varepsilon$. This means that 0 and 1 are not left-isolated, and $LP(0) = LP(1) = \varepsilon$.

	• In the 1-cycle of $G_L$ $\{2\}$, $\sigma\left(LB(\sigma(2))\right) = \varepsilon$ so, similarly, $2 \notin C_{liso}$ and $LP(2) = \varepsilon$.
	
	• In the 1-cycle of $G_L$ $\mathscr C = \{3\}$, $\sigma\left(LB(\sigma(3))\right) = 54$. We have $L(3) = \sigma\left(LB(\sigma(3))\right) = 54$, and $q_{\mathscr C}=1$ and $p_{\mathscr C}=2$ because $\sigma^{1+2}(54) = 65 = \sigma^1(54)$. Then $3 \in C_{liso}$ and $LP(3) = \sigma^2(54)~\sigma(54) = 5665$.
\end{example}

We must also take into account the fact that not every growing letter is left-periodic.

\begin{definition}
	If $c \in C$, we write $r_c \geq 0$ the first integer such that $LC\left(\sigma^{r_c}(c)\right)$ is a left-periodic letter. Note that $r_c = 0$ if $c$ is itself left-periodic.
	
	 Let $\mathscr C = \{c_i\}_{0 \leq i \leq p-1}$ be the $p$-cycle of $G_L$ such that $c_0 = LC\left(\sigma^{r_c}(c)\right)$. For $k \geq r_c+pq_{\mathscr C}$, there exists unique integers $0 \leq i \leq p-1$, $l \geq 0$ and $0 \leq l' \leq p_{\mathscr C}-1$ such that $k = r_c+i+p(q_{\mathscr C}+lp_{\mathscr C}+l')$. We then define the word
	\begin{equation*}
		LE_k(c) := \displaystyle \bigsqcup_{j=0}^{r_c+i+pl'-1} \sigma^{k-1-j}\left(LB(\sigma(LC(\sigma^j(c))))\right). 
	\end{equation*}
	Note that the words $LE_k(c)$ have bounded length because $r_c+i+pl'-1$ is bounded and $LB(\sigma(LC(\sigma^j(c)))) \in B^*$ so the words $\sigma^{k-1-j} \left(LB(\sigma(LC(\sigma^j(c))))\right)$ have bounded length.
\end{definition}

We defined the words $LE$, $LP$ and $LQ$ in order to obtain the desired decomposition.

\begin{proposition}\label{LELPLQ}
	Let $c \in C$ and let $\mathscr C = \{c_i\}_{0 \leq i \leq p-1}$ be the $p$-cycle of $G_L$ such that $c_0 = LC\left(\sigma^{r_c}(c)\right)$. For all $k = r_c+i+p(q_{\mathscr C}+lp_{\mathscr C}+l') \geq r_c+pq_{\mathscr C}$, we have
	\begin{equation*}
		LB\left(\sigma^k(c)\right) = LE_k(c) ~ LP(c_i)^l ~ LQ(c_i).
	\end{equation*}
\end{proposition}
\begin{proof}
	We first observe that, for all $j \geq 0$, $LC\left(\sigma^{r_c+j}(c)\right) = c_{j[p]}$. Then we have
	\begin{align*}
		LB\left(\sigma^k(c)\right) = & \displaystyle \bigsqcup_{j=0}^{r_c+i+pl'-1} \sigma^{k-1-j}\left(LB(\sigma(LC(\sigma^j(c))))\right) . \bigsqcup_{j=r_c+i+pl'}^{k-1} \sigma^{k-1-j}\left(LB(\sigma(LC(\sigma^j(c))))\right) \textrm{ by \Cref{LBphik}} \\
		= & ~ LE_k(c)~ \bigsqcup_{j=0}^{p(q_{\mathscr C}+lp_{\mathscr C})-1} \sigma^{p(q_{\mathscr C}+lp_{\mathscr C})-1-j}\left(LB(\sigma(LC(\sigma^{j+r_c+i+pl'}(c))))\right) \\
		= & ~ LE_k(c)~ \bigsqcup_{j=0}^{p(q_{\mathscr C}+lp_{\mathscr C})-1} \sigma^{p(q_{\mathscr C}+lp_{\mathscr C})-1-j}\left(LB(\sigma(c_{j+i[p]}))\right) \textrm{ with our initial observation} \\
		= & ~ LE_k(c)~ \bigsqcup_{j=0}^{q_{\mathscr C}+lp_{\mathscr C}-1} \bigsqcup_{j'=0}^{p-1} \sigma^{p(q_{\mathscr C}+lp_{\mathscr C})-1-pj-j'}\left(LB(\sigma(c_{j'+i[p]}))\right) \\
		= & ~ LE_k(c)~ \bigsqcup_{j=0}^{q_{\mathscr C}+lp_{\mathscr C}-1} \sigma^{p(q_{\mathscr C}+lp_{\mathscr C}-1-j)}\left(\bigsqcup_{j'=0}^{p-1} \sigma^{p-1-j'}\left(LB(\sigma(c_{j'+i[p]}))\right)\right) \\
		= & ~ LE_k(c)~ \bigsqcup_{j=0}^{q_{\mathscr C}+lp_{\mathscr C}-1} \sigma^{p(q_{\mathscr C}+lp_{\mathscr C}-1-j)}\left(L(c_i)\right) \\
		= & ~ LE_k(c)~ \bigsqcup_{j=0}^{lp_{\mathscr C}-1} \sigma^{p(q_{\mathscr C}+lp_{\mathscr C}-1-j)}\left(L(c_i)\right) . \bigsqcup_{j=lp_{\mathscr C}}^{q_{\mathscr C}+lp_{\mathscr C}-1} \sigma^{p(q_{\mathscr C}+lp_{\mathscr C}-1-j)}\left(L(c_i)\right) \\
		= & ~ LE_k(c)~ \bigsqcup_{j=0}^{l-1} \bigsqcup_{j'=0}^{p_{\mathscr C}-1} \sigma^{p(q_{\mathscr C}+lp_{\mathscr C}-1-p_{\mathscr C}j-j')}\left(L(c_i)\right) . \bigsqcup_{j=0}^{q_{\mathscr C}} \sigma^{p(q_{\mathscr C}-1-j)}\left(L(c_i)\right) \\
		= & ~ LE_k(c)~ \bigsqcup_{j=0}^{l-1} \bigsqcup_{j'=0}^{p_{\mathscr C}-1} \sigma^{p(q_{\mathscr C}+p_{\mathscr C}-1-j')}\left(L(c_i)\right) .~LQ(c_i) \\
		= & ~ LE_k(c)~ \bigsqcup_{j=0}^{l-1} LP(c_i)~.~LQ(c_i) \\
		= & ~ LE_k(c)~LP(c_i)^l~LQ(c_i).
	\end{align*}
\end{proof}

\begin{corollary}\label{uLPv}
	Let $\sigma : A \rightarrow A^+$ be a substitution. There exists a finite subset $Q_L \subset B^* \times B^*$ for which, for every $c \in C$ and every $k \geq 0$, there exist a left-periodic letter $a \in C$, $l \geq 0$ and $(u_1,u_2) \in Q_L$ such that
	\begin{equation*}
		LB\left(\sigma^k(c)\right) = u_1 ~ LP(a)^l ~ u_2.
	\end{equation*}
\end{corollary}
\begin{proof}
	Let $c \in C$ and let $\mathscr C = \{c_i\}_{0 \leq i \leq p-1}$ be the $p$-cycle of $G_L$ such that $c_0 = LC\left(\sigma^{r_c}(c)\right)$. If $k \geq r_c+pq_{\mathscr C}$, with \Cref{LELPLQ} we can write $LB\left(\sigma^k(c)\right) = u_1 ~ LP(c_i)^{l} ~ u_2$ where $u_1 = LE_k(c)$ and $u_2 = LQ(c_i)$ for some $l$ and $i$, and in that case there is a finite number of such $u_1$ and $u_2$. If $k < r_c+pq_{\mathscr C}$, we can write $LB\left(\sigma^k(c)\right) = u_2$ and there is a finite number of such $u_2$.
\end{proof}

\subsubsection{Decomposition of $RB\left(\sigma^k(c)\right)$}

This is the exact symmetric of the decomposition of $LB\left(\sigma^k(c)\right)$, so we only give the key definitions and results.

\begin{definition}\label{GR}
	In order to represent how $\sigma$ acts on the rightmost letter of $\sigma^k(c)$ for $c \in C$, we define the directed graph $G_R := \left(V_R, E_R\right)$ by
	
	• $V_R = C$,
	
	• $E_R := \left\{\left(c, RC(\sigma(c))\right) ~\middle|~ c \in C\right\}$.
\end{definition}

We have the same notions of \textit{$p$-cycle} of $G_R$ and \textit{right-$p$-periodic} letters.

\begin{example}\label{exGR}
	Following \Cref{ex0}, the graph $G_R$ is the following:
	\begin{center}
    	\begin{tikzpicture} [node distance = 2cm, on grid, auto]
        	\node (q0) [state] {0};
        	\node (q1) [state, right = of q0] {1};
        	\node (q2) [state, right = of q1] {2};
        	\node (q3) [state, right = of q2] {3};
        	\path [-stealth, thick]
            (q0) edge [bend left] node {R} (q1)
            (q1) edge [bend left] node {R} (q0)
            (q2) edge [loop right]node {R} (q2)
            (q3) edge [loop right]node {R} (q3);
		\end{tikzpicture}
	\end{center}
	The letters $0$ and $1$ are right-2-periodic, and the letters $2$ and $3$ are right-1-periodic.
\end{example}

\begin{definition}\label{Rc}
	If $\mathscr C = (c_i)_{0 \leq i \leq p-1}$ is a $p$-cycle of $G_R$, for all $i \in \llbracket 0,p-1 \rrbracket$, we define the word
	\begin{equation*}
		R(c_i) := \displaystyle \bigsqcup_{j=0}^{p-1} \sigma^j\left(RB(\sigma(c_{i+p-1-j[p]}))\right).
	\end{equation*}
\end{definition}

\begin{definition}
	Let $\mathscr C = \{c_i\}_{0 \leq i \leq p-1}$ be a $p$-cycle of $G_R$. For all $i \in \llbracket 0,p-1 \rrbracket$, $R(c_i) \in B^*$ so the sequence $\left(\sigma^{pj} (R(c_i)) \right)_{j \geq 0}$ is eventually periodic. Set $q_{\mathscr C} \geq 0$ and $p_{\mathscr C} \geq 1$ the first integers such that, for all $i \in \llbracket 0,p-1 \rrbracket$, $\sigma^{p(q_{\mathscr C}+p_{\mathscr C})}\left(R(c_i)\right) = \sigma^{pq_{\mathscr C}}\left(R(c_i)\right)$. We then define the part before the right period
	\begin{equation*}
		RQ(c) := \displaystyle \bigsqcup_{j=0}^{q_{\mathscr C}-1} \sigma^{pj}(R(c)),
	\end{equation*}
	and the right period
	\begin{equation}\label{RP}
		RP(c) := \displaystyle \bigsqcup_{j=0}^{p_{\mathscr C}-1} \sigma^{p(q_{\mathscr C}+j)}(R(c)).
	\end{equation}
	In particular, $c \in C_{riso}$ if and only if $RP(c) \neq \varepsilon$.
\end{definition}

\begin{example}\label{exRP}
	Following \Cref{exGR}, we compute $RP$ of the right-periodic letters.
	
	• Similarly to the left side, $RP(0) = RP(1) = RP(2) = \varepsilon$.
	
	• In the 1-cycle of $G_R$ $\mathscr C = \{3\}$, $\sigma(RB(\sigma(3))) = 5$. We have $R(3) = \sigma(LB(\sigma(3))) = 5$, $q_{\mathscr C}=0$ and $p_{\mathscr C}=2$. Then $3 \in C_{riso}$ and $RP(3) = 5~\sigma(5) = 56$.
\end{example}

We must also take into account the fact that not every growing letter is right-periodic.

\begin{definition}
	If $c \in C$, we write $r_c \geq 0$ the first integer such that $RC\left(\sigma^{r_c}(c)\right)$ is a right-periodic letter. Note that $r_c = 0$ if $c$ is itself right-periodic.
	
	 Let $\mathscr C = \{c_i\}_{0 \leq i \leq p-1}$ be the $p$-cycle of $G_R$ such that $c_0 = RC\left(\sigma^{r_c}(c)\right)$. For $k \geq r_c+pq_{\mathscr C}$, there exists unique integers $0 \leq i \leq p-1$, $l \geq 0$ and $0 \leq l' \leq p_{\mathscr C}-1$ such that $k = r_c+i+p(q_{\mathscr C}+lp_{\mathscr C}+l')$. We then define the word
	\begin{equation*}
		RE_k(c) := \displaystyle \bigsqcup_{j=q_{\mathscr C}+lp_{\mathscr C}}^{k-1} \sigma^j\left(RB(\sigma(RC(\sigma^{k-1-j}(c))))\right).
	\end{equation*}
	Similarly to $LE_k(c)$, the words $RE_k(c)$ have bounded length.
\end{definition}

We finally obtain the desired decomposition.

\begin{proposition}\label{RQRPRE}
	Let $c \in C$ and let $\mathscr C = \{c_i\}_{0 \leq i \leq p-1}$ be the $p$-cycle of $G_R$ such that $c_0 = RC\left(\sigma^{r_c}(c)\right)$. For all $k = r_c+i+p(q_{\mathscr C}+lp_{\mathscr C}+l') \geq r_c+pq_{\mathscr C}$, we have
	\begin{equation*}
		RB\left(\sigma^k(c)\right) = RQ(c_i) ~ RP(c_i)^l ~ RE_k(c).
	\end{equation*}
\end{proposition}
\newpage
\begin{corollary}\label{uRPv}
	Let $\sigma : A \rightarrow A^+$ be a substitution. There exists a finite subset $Q_R \subset B^* \times B^*$ for which, for every $c \in C$ and every $k \geq 0$, there exist a right-periodic letter $a \in C$, $l \geq 0$ and $(u_1,u_2) \in Q_R$ such that
	\begin{equation*}
		RB\left(\sigma^k(c)\right) = u_1 ~ RP(a)^l ~ u_2.
	\end{equation*}
\end{corollary}

\subsubsection{Decomposition of maximal bounded factors}

To recapitulate, \Cref{maxfac} states that every maximal bounded factor of $\sigma$ can be decomposed with some $LB\left(\sigma^k(c)\right)$ for $c \in C$ and $k \geq 1$, some $\sigma^k(u)$ where $u$ is part of an origin of $\sigma$ and $k \geq 1$, and some $RB\left(\sigma^k(c)\right)$ for $c \in C$ and $k \geq 1$. The words $\sigma^k(u)$ are easy to understand, and with \Cref{uLPv,uRPv} we decomposed $LB\left(\sigma^k(c)\right)$ and $RB\left(\sigma^k(c)\right)$ for all $c \in C$. We now deduce a precise decomposition of the maximal bounded factors.

\begin{proposition}\label{LXB}
	Let $\sigma : A \rightarrow A^+$ be a substitution. There exists a finite subset $Q \subset B^* \times B^* \times B^*$ such that every maximal bounded factor of $\sigma$ has the form
	\begin{equation*}
		u_1 ~ RP(a)^p ~ u_2 ~ LP(b)^q ~ u_3
	\end{equation*}
	
	where $(u_1,u_2,u_3) \in Q$, $a \in C$ is right-periodic, $b \in C$ is left-periodic and $p,q \geq 0$.
\end{proposition}
\begin{proof}
	If $u$ is a maximal bounded factor of $\sigma^k(c)$ for $c \in C$ and $k \geq 1$, \Cref{maxfac} gives three cases to consider:
	
	(i) $u = LB \left(\sigma^k(c) \right)$. By \Cref{uLPv}, there exists $(u_2, u_3) \in Q_L$, $b \in C$ left-periodic and $q \geq 0$ such that $u = u_2 ~ LP(b)^q ~ u_3$, and there is a finite number of such $u_2$ and $u_3$.
	
	(ii) $u = RB \left(\sigma^l(a) \right)~\sigma^l(v)~LB \left(\sigma^l(b) \right)$ where $0 \leq l <k$ and $(a,v,b)$ is an origin of $\sigma$. By \Cref{uRPv}, there exists $(u_1, v_1) \in Q_R$, $a \in C$ right-periodic and $p \geq 0$ such that $RB \left(\sigma^l(a) \right) = u_1 ~ RP(a)^p ~ v_1$. By \Cref{uLPv}, there exists $(v_2, u_3) \in Q_L$, $b \in C$ left-periodic and $q \geq 0$ such that $LB \left(\sigma^l(b) \right) = v_2 ~ LP(b)^q ~ u_3$. By setting $u_2 = v_1 \sigma^l(v) v_2$, we can write $u = u_1 ~ RP(a)^p ~ u_2 ~ LP(b)^q ~ u_3$, and there is a finite number of such $u_1$, $u_2$ and $u_3$ because $v \in B^*$.
	
	(iii) $u = RB \left(\sigma^k(c) \right)$. Similarly to the first case, we can write $u = u_1 ~ RP(a)^p ~ u_2$ with $a \in C$ right-periodic, and there is a finite number of such $u_1$ and $u_2$.
\end{proof}

\begin{remark}
	This proposition decomposes the maximal bounded factor of $\sigma$, which will naturally provide a decomposition of every factor of $\mathscr L(\sigma) \cap B^*$. It makes two improvements from \Cref{3forms}:
	
	(i) It explicitly defines the words that occur periodically in the decomposition.
	
	(ii) It holds not only when $\sigma$ is wild, but in the general case. As a consequence, one can recover \Cref{Lbounded}: if $\sigma$ is tame, then every $LP(c)$ and $RP(c)$ is empty so the maximal bounded factors have bounded length; if $\sigma$ is wild, then one $LP(c)$ (resp.  $RP(c)$) is non-empty so, by \Cref{LELPLQ} (resp. \Cref{RQRPRE}), the words of $\mathscr L(X_\sigma) \cap B^*$ have unbounded length.
\end{remark}

\subsection{Proof of \Cref*{main} (ii)}

We first show that the repetitions arising in \Cref{LELPLQ,RQRPRE} provide the wild minimal components exhibited in \Cref{dependance}.

\begin{proposition}\label{ntame1}
	Let $c \in C_{liso}$ (resp. $c \in C_{riso}$). Then the subshift $X({}^\omega LP(c)^\omega)$ (resp. $X({}^\omega RP(c)^\omega)$) is a wild minimal component of $X_\sigma$.
\end{proposition}
\begin{proof}
	Let $c \in C_{liso}$ be a left-$p$-periodic letter. We have $r_c = 0$ so, with \Cref{LELPLQ}, for all $l \geq 0$, $LP(c)^l \sqsubset LB\left(\sigma^{p(q_{\mathscr C}+lp_{\mathscr C})}(c)\right)$. Therefore $X({}^\omega LP(c)^\omega) \subset X_\sigma$.
	
	Symmetrically, if $c \in C_{riso}$, $X({}^\omega RP(c)^\omega) \subset X_\sigma$.
\end{proof}

The converse relies on \Cref{uLPv,uRPv}, this is a more precise version of \Cref{XBZ}.

\begin{proposition}\label{ntame2}
	Let $X$ be a wild minimal component of $X_\sigma$. Then $X$ satisfies at least one of the following properties:
	
    • There exists $c \in C_{liso}$ such that $X = X({}^\omega LP(c)^\omega)$.

    • There exists $c \in C_{riso}$ such that $X = X({}^\omega RP(c)^\omega)$.
\end{proposition}
\begin{proof}
	Let $x \in X$. Suppose that, for every $c \in C_{liso}$ (resp. $c \in C_{riso}$), there exists $k_c \geq 1$ such that $LP(c)^{k_c} \nsqsubset x$ (resp. $RP(c)^{k_c} \nsqsubset x$). Set $K$ to be the maximum of all $k_c$ for $c \in C_{liso}$ and $c \in C_{riso}$, and set $l_P := \displaystyle\max_{c \textrm{ left-periodic}} \left\lvert LP(c)^{K+2} \right\rvert + \displaystyle \max_{c \textrm{ right-periodic}} \left\lvert RP(c)^{K+2} \right\rvert$. Also set $l_Q := \displaystyle\max_{(u_1,u_2,u_3) \in Q} \lvert u_1 \rvert + \lvert u_2 \rvert + \lvert u_3 \rvert$ where $Q$ is the finite set from \Cref{LXB}. Now, for any $u \sqsubset x$, we have $u \in \mathscr L(X_\sigma) \cap B^*$ so there exists a maximal bounded factor $v$ of $\sigma$ such that $u \sqsubset v$, and, by \Cref{LXB}, there exist $(v_1,v_2,v_3) \in Q$, $a \in C$ right-periodic, $b \in C$ left-periodic and $p,q \geq 0$ such that $v = v_1~RP(a)^p~v_2~LP(b)^q~v_3$. We can then write $u = u_1 u_2 u_3 u_4 u_5$ where $u_1 \sqsubset v_1$, $u_2 \sqsubset RP(a)^p$, $u_3 \sqsubset v_2$, $u_4 \sqsubset LP(b)^q$ and $u_5 \sqsubset v_3$. We have $\lvert u_1 \rvert + \lvert u_3 \rvert + \lvert u_5 \rvert \leq \lvert v_1 \rvert + \lvert v_2 \rvert + \lvert v_3 \rvert \leq l_Q$. If $a \in C_{riso}$, by definition of $K$ we have $RP(a)^K \nsqsubset u_2$ so $\lvert u_2 \rvert < \left\lvert RP(a)^{K+2} \right\rvert$, otherwise $\lvert v_2 \rvert = 0$. Similarly, if $b \in C_{liso}$, we have $\lvert v_4 \rvert < \left\lvert LP(b)^{K+2} \right\rvert$, otherwise $\lvert v_4 \rvert = 0$. In any case, we have $\lvert u_2 \rvert + \lvert u_4 \rvert \leq l_P$, so $\lvert u \rvert \leq l_Q + l_P$. We just proved that every factor of $x$ has bounded length, which is a contradiction.

	Therefore, either there exists $c \in C_{liso}$ such that for all $k \geq 1$, $LP(c)^k \sqsubset x$, or there exists $c \in C_{riso}$ such that for all $k \geq 1$, $RP(c)^k \sqsubset x$. This means that either $X({}^\omega LP(c)^\omega) \subset X$ or $X({}^\omega RP(c)^\omega) \subset X$, and as $X$ is minimal we get the equality.
\end{proof}

\Cref{ntame1,ntame2} complete the proof of \Cref{main} (ii).

\begin{example}\label{exwcomp}
	Following \Cref{exLP,exRP}, the wild minimal components of $X_\sigma$ are $X({}^\omega (5665)^\omega)$ and $X({}^\omega (56)^\omega)$.
\end{example}

\section{Dynamics of minimal components}\label{dcomp}

We explained in \Cref{dycomp} that a substitution $\sigma$ induces a permutation $\tilde{\sigma}$ on its subshifts. In fact, one can prove that $\tilde{\sigma}$ preserves the minimal components, but we will go further by describing how $\tilde{\sigma}$ acts respectively on the tame and wild minimal components.

\subsection{Dynamics of tame minimal components}

\begin{proposition}\label{tameDynamics}
	Let $D \subset C_{niso}$ be a minimal alphabet of period $k$ and let $E$ such that $D \rightarrow E$. Then $E \subset C_{niso}$ is a minimal alphabet of period $k$ and $\tilde{\sigma}\left(X_{\sigma^k|_{D \cup B}}\right) = X_{\sigma^k|_{E \cup B}}$.
\end{proposition}
\begin{proof}
	By \Cref{Eminimal}, $E$ is a minimal alphabet of period $k$. In particular, we have $E \xrightarrow[k-1]{} D$. If $e \in E$ is a left-periodic letter, let $\mathscr C$ be the cycle of $G_L$ such that $e \in \mathscr C$ and let $d \in \mathscr C$ be such that $e \xrightarrow[k-1]{L} d$. We have $d = LC(\sigma^{k-1}(e)) \sqsubset \sigma^{k-1}(e)$ so $d \in D$. In particular, $d \notin C_{liso}$ so, by \Cref{isocycle}, $e \notin C_{liso}$. We just proved that $E \cap C_{liso} = \emptyset$, and symmetrically we get $E \cap C_{riso} = \emptyset$, therefore $E \subset C_{niso}$.
	
	Let $x \in \tilde{\sigma}\left(X_{\sigma^k|_{D \cup B}}\right)$. For all $u \in \mathscr L(x)$, it follows from the definitions that there exist $d \in D$ and $l \geq 1$ such that $u \sqsubset \sigma^{kl+1}(d)$. We also have $E \xrightarrow[k-1]{} D$ so there exists $e \in E$ such that $d \sqsubset \sigma^{k-1}(e)$. We get $u \sqsubset \sigma^{k(l+1)}(e)$ so $u \in \mathscr L \left(X_{\sigma^k|_{E \cup B}}\right)$. We just proved that $\mathscr L(x) \subset \mathscr L \left(\sigma^k|_{E \cup B}\right)$ for all $x \in \tilde{\sigma}\left(X_{\sigma^k|_{D \cup B}}\right)$, which means that $\tilde{\sigma}\left(X_{\sigma^k|_{D \cup B}}\right) \subset X_{\sigma^k|_{E \cup B}}$. Finally, by \Cref{tame1}, $X_{\sigma^k|_{E \cup B}}$ is minimal so we have the equality.
\end{proof}
	
This allows us to define the following graph.
	
\begin{definition}\label{Gt}
	We define the directed graph $G_t = \left(V_t, E_t\right)$ by
	
	• $V_t$ the set of tame minimal components of $X_\sigma$,
	
	• $E_t := \left\{(X, Y) ~\middle|~ X,Y \in V_t, \tilde{\sigma}(X)=Y \right\}$.
\end{definition}
	
Then \Cref{tameDynamics} means that $G_t$ is in correspondence with the subgraph of $G$ restricted to the minimal alphabets included in $C_{niso}$. In particular, $\tilde{\sigma}$ induces a permutation on the tame minimal components.
	
\begin{example}
	Following \Cref{exG,extcomp}, the graph $G_t$ is the following:
	\begin{center}
    \begin{tikzpicture} [node distance = 2cm, on grid, auto]
        \node (q0) [state] {$X_{\sigma^2|_{\{0,5\}}}$};
        \node (q1) [state, right = of q0] {$X_{\sigma^2|_{\{1,4,6\}}}$};
        \node (q2) [state, right = of q1] {$X_{\sigma|_{\{2,4,5,6\}}}$};
        \path [-stealth, thick]
        	(q0) edge [bend left] node {$\tilde{\sigma}$} (q1)
            (q1) edge [bend left] node {$\tilde{\sigma}$} (q0)
            (q2) edge [loop right] node {$\tilde{\sigma}$} (q2);
	\end{tikzpicture}
	\end{center}
\end{example}

\subsection{Dynamics of wild minimal components}

\begin{proposition}\label{phiLP}
	Let $a,b \in C_{liso}$ be such that $a \xrightarrow[]{L} b$. Then $\tilde{\sigma}\left(X({}^\omega LP(a)^\omega)\right) = X({}^\omega LP(b)^\omega)$.
\end{proposition}
\begin{proof}
	Let $\mathscr C = \{c_i\}_{0 \leq i \leq p-1}$ be the $p$-cycle of $G_L$ such that $a=c_0$ and $b=c_{1[p]}$.
	
	On the first hand, for all $l \geq 0$, \Cref{LELPLQ} provides
	\begin{equation*}
		LB\left( \sigma^{p(q_{\mathscr C}+lp_{\mathscr C})+1}(c_0) \right) = LE_{p(q_{\mathscr C}+lp_{\mathscr C})+1}(c_0)~LP(c_{1[p]})^l~LQ(c_{1[p]}).
	\end{equation*}
	
	On the other hand, for all $l \geq 0$, we have
	\begin{align*}
		LB\left( \sigma^{p(q_{\mathscr C}+lp_{\mathscr C})+1}(c_0) \right) = & ~LB\left( \sigma(\sigma^{p(q_{\mathscr C}+lp_{\mathscr C})}(c_0))\right) \\
		= & ~\sigma\left(LB(\sigma^{p(q_{\mathscr C}+lp_{\mathscr C})}(c_0))\right) ~ LB\left(\sigma(LC(\sigma^{p(q_{\mathscr C}+lp_{\mathscr C})}(c_0)))\right) \textrm{ by \Cref{LBRBphi}} \\
		= & ~\sigma\left(LP(c_0)\right)^l ~ \sigma\left(LQ(c_0)\right) ~ LB\left(\sigma(c_0)\right) \textrm{ by \Cref{LELPLQ}}.
	\end{align*}
	
	Both expressions are equal, and only $LP(b)^l$ on the one side and $\sigma\left(LP(a)\right)^l$ on the other side have unbounded length, which means that there exists $n \in \mathbb{Z}$ such that $T^n\left(\sigma({}^\omega LP(a)^\omega)\right) = {}^\omega LP(b)^\omega$.
\end{proof}
	
With the same arguments, we have the symmetric result. 

\begin{proposition}\label{phiRP}
	Let $a,b \in C_{riso}$ such that $a \xrightarrow[]{R} b$. Then $\tilde{\sigma}\left(X({}^\omega RP(a)^\omega)\right) = X({}^\omega RP(b)^\omega)$.
\end{proposition}

\begin{remark}
	In fact, a more precise fact holds: if $a \xrightarrow[]{L} b$ (resp. $a \xrightarrow[]{R} b$), then $LP(b)$ (resp. $RP(b)$) is a cyclic shift of $LP(a)$ (resp. $RP(a)$).
\end{remark}

This allows us to define the following graph.

\begin{definition}\label{Gw}
	We define the directed graph $G_w := \left(V_w, E_w\right)$ by
	
	• $V_w$ the set of wild minimal components of $X_\sigma$,
	
	• $E_w := \left\{(X, Y) ~\middle|~ X,Y \in V_w, \tilde{\sigma}(X)=Y \right\}$.
\end{definition}

Then \Cref{phiLP,phiRP} mean that $G_w$ is determined by $G_L$ and $G_R$. In particular, $\tilde{\sigma}$ induces a permutation on the wild minimal components.

\begin{example}
	Following \Cref{exGL,exGR,exwcomp}, the graph $G_w$ is the following:
	\begin{center}
    \begin{tikzpicture} [node distance = 3cm, on grid, auto]
        \node (q0) [state] {$X({}^\omega (5665)^\omega)$};
        \node (q1) [state, right = of q0] {$X({}^\omega (56)^\omega)$};
        \path [-stealth, thick]
        	(q0) edge [loop left] node {$\tilde{\sigma}$} (q0)
            (q1) edge [loop right] node {$\tilde{\sigma}$} (q1);
	\end{tikzpicture}
	\end{center}
\end{example}

\section{Counting minimal components}\label{numcomp}

\subsection{Computing $B$ and $C$}

We take inspiration in \cite[Lemmas 3.1 and 3.2]{Devyatov} to define the set of \textit{periodic} letters $P \subset B$ which is essential to characterize $B$.

\begin{definition}
	If $a \in B$ occurs in $\sigma^n(a)$ for some $n \geq 1$ and we say that $a$ is \textit{periodic}, otherwise we say that $a$ is \textit{pre-periodic}. We write $P$ the set of periodic letters and $PP$ the set of pre-periodic letters, such that $B = P \cup PP$.
\end{definition}

\begin{lemma}\label{caraB}
	Let $a \in A$. Then the following holds.
	
	(i) $a \in P$ if and only if there exists $n \geq 1$ such that $\sigma^n(a) = a$.
	
	(ii) If $a \in P$, then, for all $m \geq 1$, $\sigma^m(a) \in P$.
	
	(iii) $a \in B$ if and only if $\sigma^{\lvert A \rvert}(a) \in P^+$.
\end{lemma}
\begin{proof}
	(i) If $a$ is periodic, there exist $n \geq 1$ and $u,v \in A^*$ such that $\sigma^n(a) = uav$. If $(u,v) \neq (\varepsilon, \varepsilon)$, then, by iterating the substitution $\sigma^n$ on $a$, we would get $a \in C$, therefore $\sigma^n(a) = a$.
	
	If there exists $n \geq 1$ such that $\sigma^n(a) = a$, then, for all $m \geq 1$, $\sigma^{mn}(a) = a$ so $1 = \lvert \sigma^{mn}(a) \rvert \geq \lvert \sigma^m(a) \rvert \geq 1$ so $\lvert \sigma^m(a) \rvert = 1$. This means that $a \in B$, and it occurs in $\sigma^n(a)$ so $a \in P$.
	
	(ii) If $a \in P$, we already proved that, for all $m \geq 1$, $\lvert \sigma^m(a) \rvert = 1$. (i) provides $n \geq 1$ such that $\sigma^n(a) = a$, so $\sigma^n(\sigma^m(a)) = \sigma^m(\sigma^n(a)) = \sigma^m(a))$ and, by (i), $\sigma^m(a) \in P$.
	
	(iii) If $a \in B$, suppose that $\sigma^{\lvert A \rvert}(a) \notin P^+$. In particular, there exists a letter $a_{\lvert A \rvert} \in PP$ such that $a_{\lvert A \rvert} \sqsubset \sigma^{\lvert A \rvert}(a)$. This provides $a_{\lvert A \rvert-1} \sqsubset \sigma^{\lvert A \rvert-1}(a)$ such that $a_{\lvert A \rvert} \sqsubset \sigma\left(a_{\lvert A \rvert-1}\right)$, and, by (ii), $a_{\lvert A \rvert-1} \in PP$. By iterating this construction, we obtain a finite sequence $(a_i)_{0 \leq i \leq \lvert A \lvert} \in PP^{\lvert A \lvert+1}$ such that $a_{i+1} \sqsubset \sigma(a_i)$. We have $\lvert A \rvert +1$ letters so there exists $i < j$ such that $a_i = a_j$, and then $a_i \sqsubset \sigma^{j-i}(a_i)$ so $a_i \notin PP$, contradiction.
	
	If $\sigma^{\lvert A \rvert}(a) \in P^+$, in particular $\sigma^{\lvert A \rvert}(a) \in B^+$ so $a \in B$.
\end{proof}

We now have a method to compute $B$, as stated in \Cref{compB}.

\begin{proof}[\textbf{Proof of \Cref*{compB}}]
	We first compute the alphabet $P$: for $a \in A$, we compute the $\sigma^n(a)$ and either we reach $n$ such that $\lvert \sigma^n(a) \rvert \geq 2$, and by \Cref{caraB} (ii) $a \notin P$, or we reach $n$ for which there exists $m<n$ such that $\sigma^n(a) = \sigma^m(a) \in A$, and if $m = 0$ then, by \Cref{caraB} (ii), $a \in P$, or if $m > 0$ then $a \notin P$.
	
	We can now compute $B$: for $a \in A$, we compute $\sigma^{\lvert A \rvert}(a)$ and we check if it belongs to $P^+$ to use \Cref{caraB} (iii). The remaining letters are in $C$.
\end{proof}

\subsection{Computing $MC(\sigma)$}

Essentially, computing the minimal components is the same as counting them. We recall that $MC(\sigma)$ denotes the number of minimal components of $X_\sigma$, and we also define $TMC(\sigma)$ the number of tame minimal components and $WMC(\sigma)$ the number of wild minimal components.

\begin{proposition}\label{countt}
	Let $\sigma : A \rightarrow A^+$ be a substitution. Then $TMC(\sigma)$ is computable.
\end{proposition}
\begin{proof}
	Once \Cref{compB} provides $C$, we can compute $G$, the minimal alphabets and their period. By checking if the minimal alphabets contain a left or right-isolated letter, we obtain $TMC(\sigma)$.
\end{proof}
\newpage
\begin{proposition}\label{countw}
	Let $\sigma : A \rightarrow A^+$ be a substitution. Then $WMC(\sigma)$ is computable.
\end{proposition}
\begin{proof}
	Once \Cref{compB} provides $B$, we can compute $G_L$, $G_R$, $LP(c)$ and $RP(c)$ for every left or right-periodic letter. In order to differenciate the $X({}^\omega LP(c)^\omega)$ and $X({}^\omega RP(c)^\omega)$, we compute the primitive root of the words $LP(c)$ and $RP(c)$ and we check if one is a cyclic shift of another to remove duplicates.
\end{proof}

\Cref{countt,countw} complete the proof of \Cref{count}. For details about the implementation of these algorithms, the reader can take a look at the attached Python code.

\subsection{Bounding $MC(\sigma)$}

Let us first bound $TMC(\sigma)$ and $WMC(\sigma)$ separately.

\begin{lemma}\label{borne}
	Let $\sigma : A \rightarrow A^+$ be a substitution. Then
	
	(i) $TMC(\sigma) \leq \lvert C_{niso} \rvert$,
	
	(ii) $WMC(\sigma) \leq \lvert C_{liso} \rvert +  \lvert C_{riso} \rvert$.
\end{lemma}
\begin{proof}
	(i) Set $n$ to be the number of minimal alphabets included in $C_{niso}$. By \Cref{main}, $TMC(\sigma) = n$, and with \Cref{compdisjoint}, $n \leq \lvert C_{niso} \rvert$.

	(ii) $WMC(\sigma) \leq \lvert C_{liso} \rvert + \lvert C_{riso} \rvert$ comes directly from \Cref{main} (ii).
\end{proof}

We can then bound $MC(\sigma)$ using the size of the alphabet, as stated in \Cref{bound}.

\begin{proof}[\textbf{Proof of \Cref*{bound}}]
	By \Cref{borne}, we have $MC(\sigma) = TMC(\sigma) + WMC(\sigma) \leq \lvert C_{niso} \rvert + \lvert C_{liso} \rvert +  \lvert C_{riso} \rvert$.
	
	(i) If $B = \emptyset$, $C_{niso} = C = A$ and $C_{liso} = C_{riso} = \emptyset$ so $MC(\sigma) \leq \lvert C_{niso} \rvert = \lvert C \rvert = \lvert A \rvert$.
	
	(ii) If $B = \{b\}$, $\{{}^\omega b^\omega\}$ is the only possible wild minimal component, so $WMC(\sigma) \leq 1$. If $C_{liso} = C_{riso} = \emptyset$, then $MC(\sigma) \leq \lvert  C_{niso} \rvert \leq \lvert C \rvert = \lvert A \rvert - 1$. If $C_{liso} \neq \emptyset$ or $C_{riso} \neq \emptyset$, then $MC(\sigma) \leq \lvert  C_{niso} \rvert + 1 \leq \lvert C \rvert = \lvert A \rvert - 1$.
	
	(iii) If $\lvert B \rvert \geq 2$, $\lvert C_{liso} \rvert \leq \lvert C \rvert - \lvert C_{niso} \rvert$ and $\lvert C_{riso} \rvert \leq \lvert C \rvert - \lvert C_{niso} \rvert$ so $MC(\sigma) \leq 2\lvert C \rvert - \lvert C_{niso} \rvert \leq 2\lvert C \rvert \leq 2\lvert A \rvert -4$.
\end{proof}

The following examples show that these bounds are optimal.
\newline
\newline
(i) The bound $MC(\sigma) = \lvert C \rvert = \lvert A \rvert$ can only be reached if every singleton in $G$ is a minimal alphabet.

\begin{example}
	Let $k \geq 1$, $A = \{a_1,...,a_k\}$ and $\sigma : a_i \mapsto a_ia_i$ for all $i \in \llbracket 1,k \rrbracket$. We have $B = \emptyset$, the minimal alphabets are the 1-periodic alphabets $\{a_i\}$ and they generate the tame minimal components $\{{}^\omega a_i^\omega\}$, therefore $MC(\sigma) = \lvert A \rvert$.
\end{example}

(ii) The bound $MC(\sigma) = \lvert C \rvert = \lvert A \rvert-1$ can be reached in two ways:

• $TMC(\sigma) = \lvert A \rvert-1$ and $WMC(\sigma) = 0$

\begin{example}\label{ex3}
	Let $k \geq 2$, $A = \{b, a_1,...,a_{k-1}\}$, $\sigma : b \mapsto b, a_1 \mapsto a_i b a_i$. The minimal alphabets included in $C_{niso}$ are the $\{a_i\}$ for $i \in \llbracket 1,k-1 \rrbracket$ and they generate the tame minimal components $X_{\sigma|_{\{a_i,b\}}}$. In particular, for all $n \geq 0$, $\sigma^n(a_i) = (a_i b)^{2^n-1} a_i$ so $X_{\sigma|_{\{a_i,b\}}} = \{{}^\omega (a_i b)^\omega\}$. Therefore, $TMC(\sigma) = k-1 = \lvert A \rvert-1$. Also, $\sigma$ is tame so $WMC(\sigma) = 0$.
\end{example}

• $TMC(\sigma) = \lvert A \rvert-2$ and $WMC(\sigma) = 1$

\begin{example}\label{ex4}
	Let $k \geq 2$, $A = \{b, a_1,...,a_{k-1}\}$ and $\sigma : a_i \mapsto a_ia_i$ for all $i \in \llbracket 1,k-2 \rrbracket$, $a_{k-1} \mapsto a_{k-1}b, b \mapsto b$. We have $B = \{b\}$ and $a_{k-1} \in C_{riso}$ so the minimal alphabets included in $C_{niso}$ are the 1-periodic alphabets $\{a_i\}$ for $i \in \llbracket 1,k-2 \rrbracket$ and they generate the tame minimal components $X_{\sigma|_{\{a_i\}}} = \{{}^\omega a_i^\omega\}$ for $i \in \llbracket 1,k-2 \rrbracket$. Moreover, $\{{}^\omega b^\omega\}$ is the unique wild component, therefore $MC(\sigma) = \lvert A \rvert-1$.
\end{example}

(iii) The bound $MC(\sigma) = 2 \lvert C \rvert = 2 \lvert A \rvert - 4$ can only be reached when $\lvert B \rvert = 2$, $C_{liso} = C_{riso} = C$ and all $X({}^\omega LP(c)^\omega)$ and $X({}^\omega RP(c)^\omega)$ are distinct.

\begin{example}\label{ex5}
	Let $k \geq 3$, $A = \{a,b,c_1,...,c_{k-2}\}$ and $\sigma : a \mapsto a, b \mapsto b, c_i \mapsto a^{2i-2}bc_ia^{2i-1}b$. We have $B = \{a,b\}$, and every $c_i$ is growing, left-isolated with period 1 and right-isolated with period 1. It immediatly implies that $X_\sigma$ has no tame minimal components. Also, for all $i \in \llbracket k-2 \rrbracket$, $LP(c_i) = L(c_i) = a^{2i-2}b$ and $RP(c_i) = R(c_i) = a^{2i-1}b$ so the wild minimal components of $X_\sigma$ are the $X({}^\omega (a^ib)^{\omega})$ for all $0 \leq i \leq 2k-5$. Therefore $MC(\sigma) = 2k-4 = 2\lvert C \rvert = 2\lvert A \rvert - 4$.
\end{example}

\begin{remark}
	We can modify \Cref{ex5} step by step by making two wild minimal components equal, which decreases the number of minimal components by 1 (for example, if we change the image of $c_1$ to be $bc_1b$, then $MC(\sigma) = 2k-5$). That way, $MC(\sigma)$ reaches every number between 1 and $2\lvert A \rvert -4$.
\end{remark}

\section{Minimal components on two letters}\label{scase}

Substitutions on the binary alphabet $\{0,1\}$ already generate interesting subshifts and provide a reasonable number of cases to analyse. Note that the number of cases grows faster than exponentially with the size of the alphabet, so let us compute and display the minimal components of $X_\sigma$ for all cases on the alphabet $\{0,1\}$. We recall that, in order to generate a subshift, substitutions must have a growing letter so, without loss of generality, we assume that the letter 0 is always growing. Then the first distinction is wether 1 is growing or bounded.

\subsection{$C = \{0,1\}$}

\Cref{main} tells us that $X_\sigma$ only has tame minimal components, and \Cref{bound} (i) tells us that it has at most two. They are characterized by the minimal alphabets, and we recall from \Cref{generators} that $G$ is uniquely determined by its generators $D$ and $D'$ such that $\{0\} \rightarrow D$ and $\{1\} \rightarrow D'$. The following table displays the minimal components in each case, white cells mean that $X_\sigma$ is minimal, light gray cells mean that $X_\sigma$ is sometimes minimal and sometimes not, and gray cells mean that $X_\sigma$ is not minimal.

\begin{center}
\begin{tabular}{|l|c|c|c|c|}
	\hline
  	\diagbox{$D$}{$D'$} & $\{1\}$ & $\{0\}$ & $\{0,1\}$ \\
	\hline
	$\{0\}$ & \cellcolor{gray} $\{{}^\omega 0^\omega\}$, $\{{}^\omega 1^\omega\}$ & $\{{}^\omega 0^\omega\}$ & \cellcolor{lightgray} $\{{}^\omega 0^\omega\}$ \\
	\hline
	$\{1\}$ & $\{{}^\omega 1^\omega\}$ & \cellcolor{gray} $\{{}^\omega 0^\omega\}$, $\{{}^\omega 1^\omega\}$ & $X_\sigma$ \\
	\hline
	$\{0,1\}$ & \cellcolor{lightgray} $\{{}^\omega 1^\omega\}$ & $X_\sigma$ & $X_\sigma$ \\
	\hline
\end{tabular}
\end{center}

Let us detail the computations. Note that switching 0 and 1 brings the analysis down to six cases.

\subsubsection{$D = \{0\}$, $D' = \{1\}$}

This is one of the two cases where $X_\sigma$ has two minimal components. The orbits in $G$ are
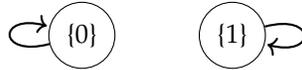
\begin{center}
    \begin{tikzpicture} [node distance = 2cm, on grid, auto]
        \node (q0) [state] {$\{0\}$};
        \node (q1) [state, right = of q0] {$\{1\}$};
        \path [-stealth, thick]
            (q0) edge [loop left] node {} (q0)
            (q1) edge [loop right] node {} (q1);
	\end{tikzpicture}
\end{center}
The minimal alphabets are the 1-periodic alphabets $\{0\}$ and $\{1\}$ so the main sub-substitutions of $\sigma$ are $\sigma|_{\{0\}}$ and $\sigma|_{\{1\}}$ and the minimal components of $X_\sigma$ are $X_{\sigma|_{\{0\}}} = \{{}^\omega 0^\omega\}$ and $X_{\sigma|_{\{1\}}} = \{{}^\omega 1^\omega\}$.

\begin{example}
	Consider the substitution $\sigma : 0 \mapsto 00, 1 \mapsto 11$.
\end{example}

\subsubsection{$D = \{1\}$, $D' = \{1\}$}

The orbits in $G$ are
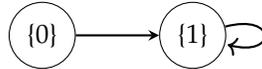
\begin{center}
    \begin{tikzpicture} [node distance = 2cm, on grid, auto]
        \node (q0) [state] {$\{0\}$};
        \node (q1) [state, right = of q0] {$\{1\}$};
        \path [-stealth, thick]
            (q0) edge node {} (q1)
            (q1) edge [loop right] node {} (q1);
	\end{tikzpicture}
\end{center}

The unique minimal alphabet is the 1-periodic alphabet $\{1\}$ so the unique main sub-substitution of $\sigma$ is $\sigma|_{\{1\}}$ and the unique minimal component of $X_\sigma$ is $X_{\sigma|_{\{1\}}} = \{{}^\omega 1^\omega\}$. In fact this case is quite trivial since $\sigma(\{0,1\}) \in \{1\}^+$, so $X_\sigma = \{{}^\omega 1^\omega\}$ and it is minimal.

\begin{example}
	Consider the substitution $\sigma : 0 \mapsto 1, 1 \mapsto 11$.
\end{example}

\subsubsection{$D = \{1\}$, $D' = \{0\}$}

This is the second case when $X_\sigma$ has two minimal components. The orbits in $G$ are
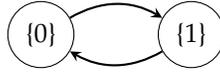
\begin{center}
    \begin{tikzpicture} [node distance = 2cm, on grid, auto]
        \node (q0) [state] {$\{0\}$};
        \node (q1) [state, right = of q0] {$\{1\}$};
        \path [-stealth, thick]
            (q0) edge [bend left] node {} (q1)
            (q1) edge [bend left] node {} (q0);
	\end{tikzpicture}
\end{center}
The minimal alphabets are the 2-periodic alphabets $\{0\}$ and $\{1\}$ so the main sub-substitutions are $\sigma^2|_{\{0\}}$ and $\sigma^2|_{\{1\}}$ and the minimal components of $X_\sigma$ are $X_{\sigma^2|_{\{0\}}} = \{{}^\omega 0^\omega\}$ and $X_{\sigma^2|_{\{1\}}} = \{{}^\omega 1^\omega\}$.

\begin{example}
	Consider the substitution $\sigma : 0 \mapsto 11, 1 \mapsto 00$.
\end{example}

\subsubsection{$D = \{0,1\}$, $D' = \{1\}$}

The orbits in $G$ are
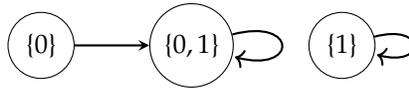
\begin{center}
    \begin{tikzpicture} [node distance = 2cm, on grid, auto]
        \node (q0) [state] {$\{0\}$};
        \node (q1) [state, right = of q0] {$\{0,1\}$};
        \node (q2) [state, right = of q1] {$\{1\}$};
        \path [-stealth, thick]
            (q0) edge node {} (q1)
            (q1) edge [loop right] node {} (q1)
            (q2) edge [loop right] node {} (q2);
	\end{tikzpicture}
\end{center}
The unique minimal alphabet is the 1-periodic alphabet $\{1\}$ so the unique main sub-substitution of $\sigma$ is $\sigma|_{\{1\}}$ and the unique minimal component of $X_\sigma$ is $X_{\sigma|_{\{1\}}} = \{{}^\omega 1^\omega\}$. In this case, there are examples where $X_\sigma$ is not minimal and others where it is minimal.

\begin{example}
	Consider the substitution $\sigma : 0 \mapsto 101, 1 \mapsto 11$. We have $\{{}^\omega 1^\omega\} \subsetneq X_\sigma = \{{}^\omega 101^\omega\}$ so it is not minimal.
\end{example}

\begin{example}
	Consider the substitution $\sigma : 0 \mapsto 01, 1 \mapsto 11$. We have $X_\sigma = \{{}^\omega 1^\omega\}$ so it is minimal.
\end{example}

\subsubsection{$D = \{0,1\}$, $D' = \{0\}$}

The orbits in $G$ are
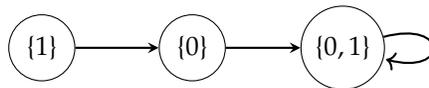
\begin{center}
    \begin{tikzpicture} [node distance = 2cm, on grid, auto]
        \node (q0) [state] {$\{1\}$};
        \node (q1) [state, right = of q0] {$\{0\}$};
        \node (q2) [state, right = of q1] {$\{0,1\}$};
        \path [-stealth, thick]
            (q0) edge node {} (q1)
            (q1) edge node {} (q2)
            (q2) edge [loop right] node {} (q2);
	\end{tikzpicture}
\end{center}

The unique minimal alphabet is the 1-periodic alphabet $\{0,1\}$ so the unique main sub-substitution is $\sigma$ itself and $X_\sigma$ is minimal.

\begin{example}
	Consider the Fibonacci substitution $\sigma : 0 \mapsto 01, 1 \mapsto 0$.
\end{example}

\subsubsection{$D = \{0,1\}$, $D' = \{0,1\}$}

The orbits in $G$ are
\begin{center}
    \begin{tikzpicture} [node distance = 2cm, on grid, auto]
        \node (q0) [state] {$\{0\}$};
        \node (q1) [state, right = of q0] {$\{0,1\}$};
        \node (q2) [state, right = of q1] {$\{1\}$};
        \path [-stealth, thick]
            (q0) edge node {} (q1)
            (q1) edge [loop above] node {} (q1)
            (q2) edge node {} (q1);
	\end{tikzpicture}
\end{center}

Similarly to the previous case, $X_\sigma$ is minimal.

\begin{example}
	Consider the Thue-Morse substitution $\sigma : 0 \mapsto 01, 1 \mapsto 10$.
\end{example}

\subsection{$C = \{0\}$, $B = \{1\}$}

In this case, $\sigma$ is l-primitive and the unique minimal alphabet is $\{0\}$. Also, \Cref{bound} (ii) tells us that $X_\sigma$ has a unique minimal component, which will only depend on wether $\sigma$ is tame or wild.

\subsubsection{$\sigma$ is tame}

The alphabet $\{0\} \subset C_{niso}$ is minimal of period 1 so, by \Cref{minimalprimitive}, $\sigma$ is l-primitive. Then, by \Cref{main} (i), $X_\sigma$ is minimal.

\begin{example}
	Consider the Chacon substitution $\sigma : 0 \mapsto 0010, 1 \mapsto 1$.
\end{example}

\subsubsection{$\sigma$ is wild}

We necessarily have $LP(0) \in \{1\}^+$ or $RP(0) \in \{1\}^+$, so, by \Cref{main} (ii), the unique minimal component of $X_\sigma$ is $\{{}^\omega 1^\omega\}$. There are examples where $X_\sigma$ is not minimal and others where it is minimal.

\begin{example}
	In \Cref{ex2}, $X_\sigma = X({}^\omega 101^\omega)$ and its unique minimal component is $\{{}^\omega 1^\omega\}$, so $X_\sigma$ is not minimal.
\end{example}

\begin{example}
	Consider \Cref{ex4} where $k=2$, that is $\sigma : 0 \mapsto 01, 1 \mapsto 1$. We have $X_\sigma = \{{}^\omega 1^\omega\}$ so it is minimal.
\end{example}

\section{Discussion}\label{discuss}

As mentionned in the introduction, the number of total subshifts would be an interesting data to evaluate how far a subshift is from being minimal. This was already discussed by Maloney and Rust in \cite[Section 5]{MR}, but we think that our approach based on giving a combinatorial characterization of the subshifts could also be useful here. As the union of subshifts is also a subshift, it is natural to consider components that are not a disjoint union of subshifts, which we might call \textit{irreducible components}. With this definition, computing the lattice of the irreducible components ordered by inclusion would be a way to describe the larger structure of the subshift. Similarly to minimal components, we would need to distinguish the irreducible components that contain growing letters and the ones that do not. Throughout our paper, \Cref{Xperiodic} can be helpful to describe the first type and \Cref{LXB} seems sufficient to describe the second type.

The question of characterizing and counting the minimal or irreducible components can also be asked for broader types of subshifts, like morphic or linearly recurrent subshifts. Notably, in the morphic case, the number of minimal components might be bigger than in the purely morphic case.

\subsubsection*{Acknowledgements}

The author thanks F. Durand, Š. Starosta and D. Rust for kindly detailing their work. He also thanks J. Cassaigne for his precious insights and the reviewers for their remarks and suggestions.

\bibliography{refs}
\bibliographystyle{plain}

\end{document}